\documentclass[11pt,oneside,english,jou]{amsart}
\usepackage[T1]{fontenc}
\usepackage[latin9]{inputenc}
\usepackage[active]{srcltx}
\usepackage{verbatim}
\usepackage{amstext}
\usepackage{amsthm}
\usepackage{amssymb}
\usepackage{fixltx2e}
 \usepackage[pdftex,bookmarks,colorlinks,breaklinks]{hyperref}

\makeatletter
\numberwithin{equation}{section}
\numberwithin{figure}{section}
\theoremstyle{plain}
\newtheorem{thm}{Theorem}[section]
  \theoremstyle{definition}
  \newtheorem{defn}[thm]{\protect\definitionname}
  \theoremstyle{definition}
  \newtheorem{rem}[thm]{\protect\remarkname}
  \theoremstyle{definition}
  \newtheorem{example}[thm]{\protect\examplename}
  \theoremstyle{plain}
  \newtheorem{prop}[thm]{\protect\propositionname}
  \theoremstyle{plain}
  \newtheorem{cor}[thm]{\protect\corollaryname}
  \theoremstyle{plain}
  \newtheorem{lem}[thm]{\protect\lemmaname}
  \theoremstyle{definition}
  \newtheorem{problem}[thm]{\protect\problemname}
  \theoremstyle{definition}
  \newtheorem{question}[thm]{Question}

\usepackage{tikz}
\usetikzlibrary{matrix}

\usepackage{xcomment}

\newcommand{\Fcal}{\mathcal{F}}

\newcommand{\Hcal}{\mathcal{HK}}

\newcommand{\Z}{\mathbb{Z}}

\newcommand{\N}{\mathbb{N}}

\newcommand{\g}{\gamma}

\newcommand{\Ga}{\Gamma}

\newcommand{\del}{\delta}
\newcommand{\Del}{\Delta}
\newcommand{\ep}{\epsilon}

\newcommand{\ol}{\overline}

\newcommand{\co}{{\rm{codim\,}}}

\newcommand{\Aut}{{\rm{Aut\,}}}

\newcommand{\Or}{{\rm{O}}}

\newcommand{\id}{{\rm{id}}}
\newcommand{\Id}{{\rm{Id}}}

\newcommand{\dist}{{\rm{dist}}}

\renewcommand{\a}{\alpha}

\renewcommand{\o}{\omega}
\newcommand{\wt}{\widetilde}

\DeclareMathOperator{\Q}{\textbf{RP}}
\DeclareMathOperator{\Qeq}{\textbf{S\textsubscript{eq}}}
\DeclareMathOperator{\PP}{\textbf{P}}
\DeclareMathOperator{\RR}{\textbf{R}}
\DeclareMathOperator{\RP}{\textbf{RP}}
\DeclareMathOperator{\NRP}{\textbf{NRP}}

\DeclareMathOperator{\Hom}{Hom}

\DeclareMathOperator{\Fix}{Fix}
\DeclareMathOperator{\codim}{codim}

\def\be{\begin{equation}}
\def\ee{\end{equation}}

\usepackage{babel}
\providecommand{\definitionname}{Definition}
  \providecommand{\examplename}{Example}
  \providecommand{\lemmaname}{Lemma}
  \providecommand{\problemname}{Problem}
  \providecommand{\propositionname}{Proposition}

\makeatother

\usepackage{babel}
  \providecommand{\corollaryname}{Corollary}
  \providecommand{\definitionname}{Definition}
  \providecommand{\examplename}{Example}
  \providecommand{\lemmaname}{Lemma}
  \providecommand{\problemname}{Problem}
  \providecommand{\propositionname}{Proposition}
  \providecommand{\remarkname}{Remark}

\begin{document}

\title[regionally proximal equivalence relation]{Higher order regionally proximal equivalence relations for general minimal group actions}

\author{Eli Glasner, Yonatan Gutman and XiangDong Ye}

\address{Eli Glasner, Department of Mathematics, Tel Aviv University, Tel
Aviv, Israel.}

\email{glasner@math.tau.ac.il}

\address{Yonatan Gutman, Institute of Mathematics, Polish Academy of Sciences,
ul. \'{S}niadeckich~8, 00-656 Warszawa, Poland.}

\email{y.gutman@impan.pl}

\address{XiangDong Ye, Wu Wen-Tsun Key Laboratory of Mathematics, USTC, Chinese
Academy of Sciences and Department of Mathematics, University of Science
and Technology of China, Hefei, Anhui, 230026, P.R. China.}

\email{yexd@ustc.edu.cn}

\keywords{Nilpotent regionally proximal relation of order $d$, enveloping semigroup,
minimal t.d.s, idempotent, equivalence relation, maximal pronilfactor,
equicontinuous.}

\subjclass[2000]{37C15, 37B20}

\maketitle

\begin{date} \today \end{date}

\newcommand*{\singlesquare}[4]{   \draw[thin,black] (0,0) node[below left] {#1} -- (1, 0) node[below right] {#2} -- (1, 1) node[above right] {#4} -- (0, 1) node[above left] {#3} -- (0,0); }
\newcommand*{\doublesquare}[7]{   \begin{scope}[xscale=#7]     \draw[thin,black] (0,0) node[below left] {#4} -- (1, 0) node[below] {#5} -- (2, 0) node[below right] {#6} -- (2, 1) node[above right] {#3} -- (1, 1) node[above] {#2} -- (0, 1) node[above left] {#1} -- (0,0);     \draw[thin,black] (1, 0) -- (1, 1);   \end{scope} }
\newcommand*{\threecube}[9]{   \begin{scope}[x={(#9, 0)}, y={(0, 1)}, z={(0.352, 0.317)}, scale=2]     \draw (0,0,0) node[below left] {#1} -- (0,0,1) node[below right] {#3} -- (0,1,1) node[above] {#7} -- (0,1,0) node[above left] {#5} -- cycle;     \draw (0,0,0) -- (1,0,0) node[below right] {#2} -- (1,0,1) node[right] {#4} -- (0,0,1) -- cycle;     \draw (1,1,1) node[above right] {#8} -- (0,1,1) -- (0,0,1) -- (1,0,1) -- cycle;
    \draw (0,0,0) -- (1,0,0) -- (1,1,0) node [above left] {#6} -- (0,1,0) -- cycle;     \draw (1,1,1) -- (1,1,0) -- (1,0,0) -- (1,0,1) -- cycle;     \draw (1,1,1) -- (0,1,1) -- (0,1,0) -- (1,1,0) -- cycle;   \end{scope} }
\begin{abstract}
We introduce higher order regionally proximal relations suitable for
an arbitrary acting group. For minimal abelian group actions, these relations
coincide with the ones introduced by Host, Kra and Maass. Our main
result is that these relations are equivalence relations whenever
the action is minimal. This was known for abelian actions by a result
of Shao and Ye. We also show that these relations lift through extensions
between minimal systems.
Answering a question by Tao, given a minimal system, we prove that the regionally
proximal equivalence relation of order $d$  corresponds to the maximal dynamical Antol\'\i n Camarena--Szegedy nilspace factor of order at most $d$. In particular the regionally proximal equivalence
relation of order one corresponds to the maximal abelian group factor. Finally by using a result of Gutman, Manners and Varj{\'u} under some restrictions on the acting group,
it follows  that the regionally proximal equivalence relation of order $d$ corresponds to the maximal pronilfactor of order at most $d$ (a factor which is an inverse limit of nilsystems of order at most $d$).
\end{abstract}
\setcounter{tocdepth}{1}
{\hypersetup{hidelinks}
\tableofcontents
}

%{\color{blue}together with the lifting property}
\section{Introduction}

\subsection{General background}

An old result in the field of topological dynamics
is a theorem by Ellis and Gottschalk \cite{EG60}, which characterizes
the equivalence relation $\Qeq(X)$, induced from the maximal equicontinuous factor of a system $(G,X)$, as the smallest $G$-invariant closed equivalence relation which contains the regionally proximal relation $\Q(X)$.
Starting with Veech \cite{veech1968equicontinuous}, various authors, including
Ellis-Keynes \cite{ellis1971characterization} and McMahon \cite{mcmahon1978relativized},
came up with various sufficient conditions for $\Q(X)$ to be an equivalence relation, whence for $\Q(X)= \Qeq(X)$. In particular, they proved that for a minimal system with an abelian acting group this is indeed the case.

Host, Kra and Maass \cite{HKM10} introduced the higher
order regionally proximal relations $\RP^{[d]}(X)$ ($d\in\mathbb{{N}}$, $\RP^{[1]}(X)=\Q(X)$)
for abelian actions, while investigating a topological dynamical analog
of the celebrated Host-Kra structure theorem \cite{HK05}. One of
their results was that $\RP^{[d]}(X)$ are equivalence relations for
minimal distal systems $(\mathbb{{Z}},X)$. Shao and Ye \cite{SY12}
generalized this theorem and showed
that the relations
$\RP^{[d]}(X)$ are equivalence relations for all minimal actions by abelian groups.

In this article we present a new definition, the \textit{nilpotent regionally proximal relations} of order $d$, $\NRP^{[d]}(X)$ $(d\in\mathbb{{N}})$,
defined for general group actions $(G,X)$.
These are closed and $G$-equivariant
relations which coincide with the Host-Kra-Maass definition for minimal abelian
group actions. However, for non-abelian group actions it may happen
that $\Q(X)\subsetneq\NRP^{[1]}(X)$.

Our main result is that
for minimal actions $\NRP^{[d]}(X)$ is an equivalence relation for
all $d\in\mathbb{{N}}$ . This result is surprising as the regionally
proximal relation $\Q(X)$ is known \textit{not} to be an equivalence
relation for some (non-amenable) group actions (\cite{mcmahon1976}).

The proof of Shao and Ye for abelian group actions was based on the
general structure theory of minimal actions due to Ellis-Glasner-Shapiro
\cite{ellis1975proximal}, McMahon \cite{mcmahon1978relativized}
and Veech \cite{V77}. In this article we present a direct enveloping
semigroup proof of this theorem which is very similar to the short
proof by Ellis and Glasner of the celebrated theorem by Van der Waerden
on the existence of arbitrary long monochromatic arithmetic progressions
in finite colorings of the integers (\cite{G03,G94}). The proof is
shorter and yields the result for general group actions. The possibility
of applying the Ellis-Glasner proof as a shortcut to Shao and Ye's
proof in the abelian setting was also discovered by Ethan Akin (\cite{AkinPrivate}).

Generalizing a result of Shao and Ye in the abelian setting (\cite{SY12}),
we show that given an extension of minimal systems $\pi:(G,X)\to(G,Y)$,
the nilpotent regionally proximal relation lifts, i.e. $\pi\times\pi(\NRP^{[d]}(X))=\NRP^{[d]}(Y)$.
From this one easily concludes that for any minimal system, $(G,X/\NRP^{[d]}(X))$
is the maximal factor of $(G,X)$ for which the nilpotent regionally proximal
relation of order $d$ is trivial. Following \cite{HKM10}, we call such systems
\textit{systems of order at most $d$}. By the theory developed
by Gutman, Manners and Varj{\'u} in \cite{GMVIII}, it follows that a system $(G,X)$ of order at most $d$ , where $G$ has a dense subgroup generated by a compact group, is a \textit{pronilsystem}
of order at most $d$, that is an inverse limit of nilsystems of
order at most $d$. Nilsystems, pronilsystems and the related \textit{nilsequences
}appear in different guises in several areas of mathematics: topological
dynamics (\cite{auslander1963flows}), ergodic theory (\cite{HK05,Z07}),
additive number theory (\cite{GT10}) and additive combinatorics (\cite{S12}).

The paper \cite{GMVIII} forms the third part of a series by the same authors \cite{GMVI,GMVII} extending the ground-breaking work of Antol\'\i n Camarena and Szegedy \cite{CS12}, where the concept of  \emph{nilspaces} was introduced. A nilspace is a compact space $X$ together with closed collections of \emph{cubes} $C^n(X)\subseteq X^{2^n}$, $n=1,2,\ldots$, satisfying some natural axioms. We show $(G,X/\NRP^{[d]}(X))$ equipped with a natural collection of cubes is the maximal factor of $(G,X)$ which is a  nilspace of order at most $d$. This answers a question by Tao in \cite{TaoBlog15}.

Comparing $\Qeq(X)$, the smallest equivalence relation which contains
the regionally proximal relation $\Q(X)$ with $\NRP^{[1]}(X)$, we
show that while the former corresponds to the maximal equicontinuous
factor, the latter corresponds to the maximal (compact) abelian group
factor. Thus unlike in the case of the maximal equicontinuous factor
we have an explicit and unknown hitherto form for the equivalence
relation corresponding to the maximal abelian group factor for arbitrary
minimal actions. One may wonder whether a similar result can be achieved for
the maximal (not necessarily abelian) group factor of a general minimal
system.

\subsection{Structure of the paper}

Section \ref{sec:Preliminaries} contains basic notation. Section
\ref{sec:The-higher-order} introduces the nilpotent higher order regionally
proximal relations. In Section \ref{sec:minimal-subsystems-for} we
prove several results that play a key role in Section \ref{sec:RPd is an equivalence relation}.
Section \ref{sec:RPd is an equivalence relation} is devoted to proving
the main result of the paper, namely that the nilpotent higher order regionally
proximal relations are equivalence relations for general minimal group
actions. In Section \ref{sec:Lifting--from} we show that the nilpotent regionally
proximal equivalence relations lift through dynamical morphisms between
minimal systems. In Section \ref{sec:systems-of-order} we investigate
the structure of systems whose nilpotent regionally proximal equivalence relation
of order $d$ is trivial and answer Tao's question. In Subsection \ref{sec:The-relation-between}
we investigate the relation between the classical regionally proximal
relation and the nilpotent regionally proximal equivalence relation of order
one. In Subsection \ref{sec:diff gen} we present a different higher order generalization of the classical regionally proximal relation for arbitrary group actions, about which we know little.
In Section \ref{sec:A-dynamical-cubespace} we exhibit an example
related to Section \ref{sec:systems-of-order}. Section \ref{sec:Open-Questions} is dedicated to open questions. Finally the Appendix contains technical
results.

\subsection{Acknowledgements}
%011116

The first author was partially supported by
a grant of the Israel Science Foundation (ISF 668/13).
The second author was partially supported by the Marie Curie grant PCIG12-GA-2012-334564
and by the NCN (National Science Center, Poland) grant 2016/22/E/ST1/00448.
%{\color{blue}
The third author was partially supported by NNSF of China 11371339  and 11431012.
%}.
The work originated in the trimester program
on ``Universality and Homogeneity'' at the Hausdorff Research Institute
for Mathematics in Bonn where the first and second authors took part.
We are grateful to the organizers of the program, Alexander Kechris,
Katrin Tent and Anatoly Vershik. A significant part of the work was carried out during the visit of the second author to the third author at the University of Technology and Science of China at Hefei in October 2015. The second author is grateful for the hospitality and excellent working conditions during the visit. The work was concluded in the Simons Semester ``Dynamical Systems'' at the Banach Center in Warsaw (September-December 2015) where the first and second authors took part. We acknowledge the partial support by  grant 346300 for IMPAN from the Simons Foundation and the matching 2015-2019 Polish MNiSW fund. We are grateful to Yixiao Qiao for help with Figure \ref{fig:circle}. We thank Joe Auslander for sending us comments on a previous version. Finally, we are grateful to the referee for a careful reading and many useful suggestions.

\section{Preliminaries\label{sec:Preliminaries}}

\subsection{The underlying system }

Throughout the article, $(G,X)$ denotes a \textbf{topological dynamical
system} (t.d.s) where $G$ is a (Hausdorff) topological group and $X$
is a compact Hausdorff space. To improve readability we sometimes assume without loss of generality that $X$ is metrizable and use sequences of elements instead of nets of elements. We stress that this assumption is superfluous unless stated explicitly\footnote{In fact the only results which require that $X$ is metrizable are Theorem \ref{thm:genericity} and Theorem \ref{thm:strong structure}.}. When a metric is evoked we denote it by $\dist()$.  The action of an element $g\in G$ on $x\in X$ is denoted by $gx$. For $x\in X$, $\Or(x,G)=Gx$ denotes the orbit of $x$. A continuous
$G$-equivariant map $(G,X)\to(G,Y)$ is called a \textbf{dynamical morphism}\footnote{We do not insist a dynamical morphism to be surjective, however most (but not all) dynamical morphisms appearing in this article are.}. In exceptional and explicitly stated cases, we allow for dynamical morphisms between two t.d.s $(G,X)$ and $(G',X')$, where possibly $G\neq G'$. In such a case there  exist a continuous group homomorphism $\phi:G\to G'$  and a continuous map $f:X\to X'$ such that for all $x\in X$ and $g\in G$, $f(gx)=\phi(g)f(x)$.

Discrete cubes and their faces will appear abundantly throughout the
article. In the next subsections we summarize some related notation.

\subsection{Discrete cubes\label{sub:Discrete-cubes}}

For an integer $d\ge0$ we denote the set of maps $\{0,1\}^{d}\to X$
by $X^{[d]}$ ($X^{[0]}=X$) and call its elements $d$-\textbf{configurations}.
For a configuration $x\in X^{[d]}$, we call the points $\{x(\o)\}_{\o\in\{0,1\}^{d}}$
the \textbf{vertices} of $x$. Let $\pi_{*}:X^{[d]}\to X^{\{0,1\}_{*}^{d}}\triangleq X^{\{0,1\}^{d}\setminus\{\vec{{0}}\}}$
denote the projection onto the last $(2^{d}-1)$-coordinates; i.e.,
the map which forgets the $\vec{0}$-coordinate. Let $X_{*}^{[d]}=\pi_{*}(X^{[d]})=\prod\{X_{\ep}:\ep\not=\vec{0}\}$
and for $x\in X^{[d]}$ let $x_{*}=\pi_{*}(x)\in X_{*}^{[d]}$ denote
its projection. Sometimes it is convenient to write $x=(x_{\vec{0}},x_{*})$.
For each $\ep\in\{0,1\}^{d}$ we denote by $\pi_{\ep}$ the projection
map from $X^{[d]}$ onto $X_{\ep}=X$. For a point $x\in X$ we let
$x^{[d]}\in X^{[d]}$ and $x_{*}^{[d]}\in X_{*}^{[d]}$ be the \textbf{constant}
configuration, that is, the configuration all of whose
vertices are equal to $x$. We denote by
$\Del^{[d]}=\Del^{[d]}(X)=\{x^{[d]}:x\in X\}$, the \textbf{diagonal}
of $X^{[d]}$ and by $\Del_{*}^{[d]}=\Del_{*}^{[d]}(X)=\{x_{*}^{[d]}:x\in X\}$
the diagonal of $X_{*}^{[d]}$. Sometimes it is convenient to represent
$X^{[d]}$ $(d\geq1)$ as a product space $X^{[d]}=X^{[d-1]}\times X^{[d-1]}$.
When using this decomposition we write $x=(x_{f},x_{c})$ and refer
to $x_{f}$ and $x_{c}$ as the \textbf{floor} and \textbf{ceiling}
of $x$. More explicitly define the identification $X^{[d]}\to X^{[d-1]}\times X^{[d-1]}$
by $x\mapsto(x_{f},x_{c})$ with $(x_{f})_{\ep}=x_{\ep0}$ and $(x_{c})_{\ep}=x_{\ep1}$
for $\epsilon\in\{0,1\}^{d-1}$. If $f',f'':X^{[d-1]}\to X^{[d-1]}$
are functions then we define

\begin{equation}
f'\times f''(x_{f},x_{c})=(f'(x_{f}),f''(x_{c}))\label{eq:product}
\end{equation}
and define $\pi_{f},\pi_{c}:X^{[d]}\to X^{[d-1]}$ by $\pi_{f}(x)=x_{f}$
and $\pi_{c}(x)=x_{c}$. Further in the case of $X^{[2]}$ we will
employ the following identification:

\begin{equation}\label{eq:X2 identification}
X^{[2]}\to X\times X\times X\times X\qquad x\mapsto(x_{00},x_{10},x_{01},x_{11})
\end{equation}

\subsection{Faces\label{sub:Faces}}

Let $d\geq0$. A set of the form
$F=\{\o\in\{0,1\}^{d}|\ \o_{i_1}=\a_{1},\o_{i_{2}}
=\a_{2},\ldots,\o_{i_{k}}=\a_{k}\}$ for some $k\geq0$, $1\le i_{1}<i_{2}<\cdots<i_{k}\le d$ and $\a_{i}\in\{0,1\}$ is called a \textbf{face }of \textbf{codimension }$k$ of the discrete
cube $\{0,1\}^{d}$.\footnote{The case $k=0$ corresponds to $\{0,1\}^{d}$.}
One writes $\codim(F)=k$. A face\textbf{ }of codimension\textbf{
}$1$ is called a \textbf{hyperface}. If all $\a_{i}=1$ we say that
the face is \textbf{upper}. Note all upper faces contain $\vec{1}$
and there are exactly $2^d$ upper faces. Similarly if all $\a_{i}=0$
we say that the face is \textbf{lower}.

\begin{figure}[h]
\centering{}\begin{tikzpicture}
 \begin{scope}[x={(1, 0)}, y={(0, 1)}, z={(0.352, 0.317)}, scale=3]
  \draw (0,0,0) node[below left] {$x_{000}$} -- (0,0,1) node[above left] {$x_{010}$} -- (0,1,1) node[above]
  {$x_{011}$} -- (0,1,0) node[above left] {$x_{001}$} -- cycle;
  \draw (0,0,0) -- (1,0,0) node[below right] {$x_{100}$} -- (1,0,1) node[right] {$x_{110}$} -- (0,0,1) -- cycle;
  \draw (1,1,1) node[above right] {$x_{111}$} -- (0,1,1) -- (0,0,1) -- (1,0,1) -- cycle;
  \draw (1,1,1) -- (1,1,0) -- (1,0,0) -- (1,0,1) -- cycle;           \draw (1,1,1) -- (0,1,1) -- (0,1,0) -- (1,1,0) -- cycle;
  \draw[fill=black!50, opacity=.8] (0,0,0) -- (0,0,1) -- (1,0,1) -- (1,0,0) -- cycle;
  \draw (0,0,0) -- (1,0,0) -- (1,1,0) node [above left] {$x_{101}$} -- (0,1,0) -- cycle;
   \end{scope}
\end{tikzpicture} \caption{A $3$-configuration with a shaded lower hyperface.}
\end{figure}
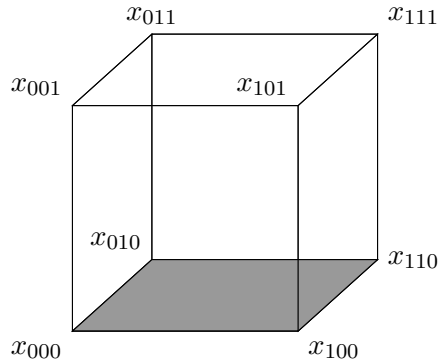

\section{Nilpotent regionally proximal relations\label{sec:The-higher-order}}

\subsection{Proximality and its generalizations\label{sub:Proximality-and-its}}
Let us recall several classical definitions.
Two points $x,y\in X$ are said to be\textbf{ proximal}, denoted $(x,y)\in\PP(X)$,
if there is a sequence of elements $g_{i}\in G$ such that $\lim_{i\to\infty}\dist(g_{i}x,g_{i}y)\allowbreak =0$.
The system $(G,X)$ is said to be \textbf{distal}, if
%{\color{blue}
$\PP(X)=\Del=\{(x, x)|x \in X\}$.
%}
Two points $x,y\in X$ are said to be \textbf{regionally proximal},
denoted $(x,y)\in \Q(X)$, if there are sequences of points $x_{i},y_{i}\in X$
and a sequence of elements $g_{i}\in G$ such that $\lim_{i\to\infty}x_{i}=x$,
$\lim_{i\to\infty}y_{i}=y$ and $\lim_{i\to\infty}\dist(g_{i}x_{i},g_{i}y_{i})=0$.
Let $\Qeq(X)$ be the smallest $G$-invariant closed equivalence relation which contains $\Q(X)$\footnote{By \cite[Theorem 2]{EG60}, $\Qeq(X)$ is induced from the maximal equicontinuous factor of $(G,X)$.}. Clearly $\PP(X)\subset\Q(X)\subset\Qeq(X)$.
It is a remarkable fact that in many cases the regionally proximal relation
happens to be an equivalence relation, i.e it holds $\Q(X)=\Qeq(X)$.
These cases include,
inter alia, the case when $(G,X)$ is proximal or weakly mixing,
or when it is minimal
and admits an invariant measure (\cite{mcmahon1978relativized}, see \cite[Theorem 9.8]{A}).
In particular if $(G, X)$ is
minimal with $G$ amenable, or when it is minimal and satisfies the Bronstein condition\footnote{$(G,X)$ is said to satisfy the Bronstein condition if $X\times X$
has a dense set of minimal points. The mentioned result was proven
in \cite[Theorem 2.6.2]{V77} and was also obtained independently
by Ellis (unpublished).}.
A particular case of the latter occurs when $(G, X)$ is minimal and point-distal
(\cite[Proposition 16.10]{ellis2014automorphisms},
first proven in \cite{ellis1971characterization}).
It is also known
that the regionally proximal relation can fail to be an equivalence
relation for minimal t.d.s. A well known counterexample is given in \cite[Example 1.8]{mcmahon1976}
(for more details see \cite[V(1.8)(2)]{dV93}).

In \cite{HKM10} Host, Kra and Maass introduced the regionally proximal
relation of order $d$ $(d\in\mathbb{{N}})$ for $G$ abelian, where
the case $d=1$ corresponds to the classical regionally proximal relation:
\begin{defn}\label{HK-def}
\cite[Definition 3.2]{HKM10} Let $(G,X)$ be a topological dynamical
system with $G$ abelian and $d\in\mathbb{{N}}$. The points $x,y\in X$
are said to be \textbf{regionally proximal of order $d$}, denoted
$(x,y)\in\RP^{[d]}(X)$, if there are sequences of elements $g_{i}^{1},g_{i}^{2},\ldots,g_{i}^{k}\in G,\ x_{i},y_{i}\in X$,
such that for all $(\epsilon_{1},\epsilon_{2},\ldots\epsilon_{d})\in\{0,1\}_{*}^{d}$:
\[
\lim_{i\to\infty}x_{i}=x,\ \lim_{i\to\infty}y_{i}=y,\:\lim_{i\to\infty}\dist\big((\sum_{j=1}^{d}\epsilon_{j}g_{i}^{j})x_{i},(\sum_{j=1}^{d}\epsilon_{j}g_{i}^{j})y_{i}\big)=0
\]

\end{defn}
In order to generalize this definition to general group actions, we
introduce several important concepts in the next two subsections.

\subsection{Host-Kra cube group\label{sub:Host-Kra-cube-group}}

Let $H\subset G$ be a subgroup and $F\subset\{0,1\}^{d}$. For $h\in H$
we denote by $[h]_{F}$ the element of
$H^{[d]}= H^{\{0,1\}^d}$
defined as $[h]_{F}(\o)=h$
if $\o\in F$ and $[h]_{F}=\Id$ otherwise, where $\Id$ denotes the
unit element of $G$. Define:
\[
[H]_{F}=\{[h]_{F}|h\in H\}
\]
We call the subgroup of $G^{[d]}$ generated by $[G]_{F}$, where
$F$ ranges over all \textit{hyperfaces} of $\{0,1\}^{d},$ the \textbf{Host-Kra
cube group}\footnote{The terminology is due to \cite[Definition E.3]{GT10} where it is
employed in the context of filtered Lie groups. } and denote it by $\Hcal^{[d]}$. We call the subgroup of $G^{[d]}$
generated by $[G]_{F}$ where $F$ ranges over all \textit{upper hyperfaces} of $\{0,1\}^{d}$ the \textbf{face cube group} and denote it by $\Fcal^{[d]}$. It is easy to see that $\Hcal^{[d]}$ is generated by $\Fcal^{[d]}$ and $\Del^{[d]}(G)$.
\begin{example}
One sees readily that $\Fcal^{[2]}$ is generated by $$\{(\Id,\Id,h,h), (\Id,h',\Id,h'): h,h'\in G\}$$ and
$\Hcal^{[2]}$ is generated by $$\{(\Id,\Id,h,h), (\Id,h',\Id,h'), (k,k,\Id,\Id), (k',\Id,k',\Id): h,h',k,k'\in G\}.$$
Thus, $\Hcal^{[2]}$ is generated by $\Fcal^{[2]}$ and $\{(t,t,t,t): t\in G\}=\Del^{[2]}(G).$
\end{example}

The Host-Kra and face cube groups originate in \cite[Section 5]{HK05} and coincide with
the \textit{parallelepiped groups} and \textit{face groups} respectively
of \cite[Definition 3.1]{HKM10} introduced for abelian actions. Notice
$\Fcal^{[d]}\subset\Hcal^{[d]}$ and for all $\g\in\Fcal^{[d]}$,
$\g(\vec{0})=\Id$. The Host-Kra cube and face groups act (coordinate-wise)
on $X^{[d]}$ by $\g c(\o)=\g(\o)c(\o)$ for $\g\in\Hcal^{[d]}$,
$c\in X^{[d]}$ and $\o\in\{0,1\}^{d}$. Similarly the face group act (coordinate-wise)
on $X_{*}^{[d]}$.

\begin{prop}
\label{prop:H=00003DFDiag}Let $G$ be a group then $\Hcal^{[d]}=\Fcal^{[d]}\Del^{[d]}(G)$.\end{prop}
\begin{proof}
By definition of the groups involved it is easy to see $\Fcal^{[d]}\Del^{[d]}(G)\subset\Hcal^{[d]}$.
In order to prove the reverse direction fix $g\in\Hcal^{[d]}$ where
by definition $g=\prod_{j=1}^{m}[t_{j}]_{F_{j}}$ where $F_{j}$ is
an upper hyperface or $F_{j}=\{0,1\}^{d}$ and $t_{j}\in G$.
%{\color{red} (this does not follow from the definition and we need to add a sentence to explain)}
Note that for $F_{j}$ an upper hyperface one has:

\[
[t_{1}]_{\{0,1\}^{d}}[t_{2}]_{F_{j}}=[t_1 t_{2}t_1^{-1}]_{F_{j}}[t_{1}]_{\{0,1\}^{d}}.
\]

This implies that one can move all occurrences of the form $[t_{1}]_{\{0,1\}^{d}}$
to the right while leaving on the left only expressions of the form

 $[t]_{F_{j}}$ with $F_j$ an upper hyperfaceface.

Thus we have shown $g\in\Fcal^{[d]}\Del^{[d]}(G)$.
\end{proof}

\subsection{Dynamical cubespaces\label{sub:Dynamical-cubespaces}}

The notion of cubespaces and nilspaces originate from Host and Kra's
{\em parallelepiped structures} in \cite{HK08}. Antol\'\i n Camarena and Szegedy
\cite{CS12} carried out a systematic study of nilspaces and described
their structure. We recommend \cite[Subsection 1.3]{GMVIII} for a
succinct introduction to nilspaces (but see also \cite{GMVI,GMVII,can17algebraic,can17compact}).
In this subsection we introduce the notion of \textit{dynamical cubespaces}.
For the general theory see Subsection \ref{sub:Nilspaces}. We stress that the dynamical cubespaces are a subclass of the class of cubespaces (see Proposition \ref{prop ap:cube inv}).

Let $(G,X)$ be a topological dynamical system. Following Host, Kra
and Maass \cite[Definition 1.1]{HKM10} we introduce the \textbf{dynamical
cubes} as the orbit closure of constant configurations and denote
it by
\begin{equation}
C_{G}^{[d]}(X)=\overline{\{gx^{[d]}|g\in\Hcal^{[d]},x\in X\}}\quad(d\geq0).\label{eq:def of Cn H}
\end{equation}

The pair $(X,{C_{G}^{\bullet}})\triangleq(X,{\{C_{G}^{[d]}(X)\}_{d\in\mathbb{{Z}}_{+}}})$
is called a \textbf{dynamical cubespace} induced by $(G,X)$. We also
denote:

\[
C_{x}^{[d]}(X)=C_{G}^{[d]}(X)\cap\big(\{x\}\times X_{*}^{[d]}\big)
\]

\[
C_{x*}^{[d]}(X)=\pi_{*}(C_{x}^{[d]}(X))
\]

\begin{prop}\label{prop:face rep for Cd}
Let $(G,X)$ be a minimal t.d.s, $x_{0}\in X$
and $d\geq1$ then
\end{prop}
\[
C_{G}^{[d]}(X)=\overline{\{gx_{0}^{[d]}|\ g\in\Hcal^{[d]}\}}=\overline{\{gx^{[d]}|\ g\in\Fcal^{[d]},x\in X\}}
\]

\begin{proof}
The first equality is trivial. The second follows from Equation (\ref{eq:def of Cn H})
and Proposition \ref{prop:H=00003DFDiag}.
\end{proof}

\subsection{Nilpotent regionally proximal relations\label{sub:Regionally-proximal-equivalence}}

We are ready to introduce the definition of the nilpotent regionally proximal relations
for general group actions.
%011116
%{\color{blue}
\begin{defn}\label{NC-def}
Given $x,y\in X$
% are two points,
we let
% us introduce
the \textbf{lower
corner} $\llcorner^{[d]}(x,y)$
%for
be
the configuration defined by:
$\o\mapsto x$
for $\o\neq\vec{1}$ and $\vec{1}\mapsto y$;
and the \textbf{upper corner} $\urcorner^{[d]}(x,y)$ by the configuration: $\vec{0}\mapsto x$
and $\o\mapsto y$ for $\o\neq\vec{0}$.
\end{defn}

\begin{defn}
\label{def:general RPd}Let $(G,X)$ be a topological dynamical system. Let $d\geq1$.
We say that a pair of points, $x,y\in X$ are \textbf{nilpotent regionally proximal
of order} $d$ and write $(x,y)\in\NRP^{[d]}(X)$, if and only if $\llcorner^{[d+1]}(x,y)\in C_{G}^{[d+1]}(X)$. That is $(x,y)\in\NRP^{[d]}(X)$ if and only if
there are sequences $g_i \in  \Hcal^{[d+1]}$ and $x_i \in X$
such that
$$
\lim_{i\to \infty} g_i x_i^{[d+1]} = \llcorner^{[d+1]}(x,y).
$$
\end{defn}

%We note that sometimes when using induction it is convenient to define $\NRP^{[0]}(X)=X\times X$. To have a better understanding of the definition, we have the following remarks.

Host, Kra and Maass \cite[Corollary 4.3]{HKM10}
showed that if $(\mathbb{{Z}},X)$ is minimal and distal, then $(x,y)\in\RP^{[d]}(X)$
if and only if $\llcorner^{[d+1]}(x,y)\in C_{\mathbb{{Z}}}^{d+1}(X)$.
This was generalized to arbitrary minimal abelian actions by Shao
and Ye in \cite[Theorem 3.4]{SY12}. Thus for minimal abelian actions $(G,X)$, $\NRP^{[d]}(X)=\RP^{[d]}(X)$ \footnote{It is easy to see that this statement is not true in general if we remove the minimality assumption.}.

When $G$ is abelian there are, canonically defined, surjective, group homomorphisms
$G^d = G \times G \times \cdots \times G \to  \Fcal^{[d]}$
and $G^{d+1} = G \times G \times \cdots \times G \to  \Hcal^{[d]}$,
namely,
$$
(g_1, g_2, \dots, g_d) \in G^d \mapsto
\big(\sum_{j=1}^{d}\epsilon_{j}g_{j} |
(\epsilon_{1},\ldots,\epsilon_{d}) \in \{0,1\}^{d}\big)
$$
and
$$
(g_1, g_2, \dots, g_d, g_{d+1}) \in G^{d+1} \mapsto
\big(g_{d+1} + \sum_{j=1}^{d}\epsilon_{j}g_{j} |
(\epsilon_{1},\ldots,\epsilon_{d}) \in\{0,1\}^{d}\big),
$$
respectively.
This fact explains why  Host and Kra's
definition of $\RP^{[d]}(X)$ (Definition \ref{HK-def}) is much simpler than Definition
\ref{NC-def}. However, as we will see, when the action is given by a non-commutative $G$, commutation relations in $G$ or, more precisely, its lower central series,
determine, through the cubic structure, the behaviour of the $\NRP^{[d]}(X)$ relations.

The reader may wonder why the word "nilpotent" appears in the name of $\NRP^{[d]}(X)$. The reason is that if $\NRP^{[d]}(X)$ is trivial, i.e. $\NRP^{[d]}(X)=\triangle$, then $(G,X)$ is isomorphic to an action by a nilpotent group of nilpotency class at most $d$ (for an exact statement see Proposition \ref{prop:effective action}). Another natural question is what is the relation between $\NRP^{[1]}(X)$ and $\PP(X),\Q(X),\Qeq(X)$ defined in Subsection \ref{sub:Proximality-and-its}. It turns out that  $\PP(X)\subset\Q(X)\subset\Qeq(X)\subset\NRP^{[1]}(X)$  (for a proof see Proposition \ref{prop:Q subset RP1}). Finally we remark that we could have used the upper corner $\urcorner^{[d+1]}(x,y)$ in the definition of $\NRP^{[d]}(X)$ as  by Proposition \ref{prop:alt def of nrp}, $(x,y)\in\NRP^{[d]}(X)$ if and only if $\urcorner^{[d+1]}(x,y)\in C_{G}^{[d+1]}(X)$.

\begin{example}\label{ex:nrp}
Fix $d\in\mathbb{N}$. We give two examples of calculating $\NRP^{[d]}(X)$ of very different in flavor. Let $G_{d+1}$ denotes the $(d+1)$-th element
of the lower central series of $G$ (see Subsection \ref{sub:Lower-central-series}),
then $(x,hx)\in\NRP^{[d]}(X)$ for any $h\in G_{d+1}$. (For a proof
of this fact, see Lemma \ref{lem:elementary}(4)). Hence if $G$ is
perfect, that is $G=[G,G]$, and the action is minimal, then $\NRP^{[d]}(X)=X^{[d]}$.
For the second example let $(G,L/\Ga)$ be a (generalized) minimal nilsystem,
that is, $L$ is a nilpotent Lie group of nilpotency
class at most $d$, $\Ga\subseteq L$ is a discrete cocompact subgroup
and the minimal action of $G$ on $L/\Ga$ is through a continuous group homomorphism $\phi:G\to L$.
In \cite[Proposition 2.5]{GMVI} based on \cite[Lemma E.10]{GT10} it is proven that
$\{g\Gamma^{[d+1]}|\ g\in\Hcal^{[d+1]}\}$ is compact. By Proposition \ref{prop:face rep for Cd} we conclude $C_{G}^{[d+1]}(X)=\{g\Gamma^{[d+1]}|\ g\in\Hcal^{[d+1]}\}$. By \cite[Proposition 2.6]{GMVI} if $c,c'\in C_{G}^{[d+1]}(X)$ such that  $c(\o)=c'(\o)$
for $\o\neq\vec{1}$ then $c(\vec{1})=c'(\vec{1})$. It follows that if $\llcorner^{[d+1]}(x,y)\in C_{G}^{[d+1]}(X)$  then $x=y$ as $x^{[d+1]}\in C_{G}^{[d+1]}(X)$. We conclude
$\NRP^{[d]}(L/\Ga)$ is trivial. See also Subsection \ref{subsec:Overview}.
\end{example}

Shao and Ye \cite[Theorem 3.5]{SY12} showed that $\RP^{[d]}(X)$
is an equivalence relation for minimal actions by abelian groups. Our main result is the following theorem:
\begin{thm}
\label{thm:main}Let $(G,X)$ be a minimal topological dynamical system, then $\NRP^{[d]}(X)$
$(d\geq1)$ is a closed $G$-invariant equivalence relation.
\end{thm}

The theorem is surprising as the regionally proximal relation $\Q(X)$ is known \textit{not} to be an equivalence relation for some (non-amenable) group actions (\cite[Example 1.8]{mcmahon1976}; for more details see \cite[V(1.8)(2)]{dV93}).

\section{Minimal subsystems for the Host-Kra and face cube groups \label{sec:minimal-subsystems-for}}

Let $(G,X)$ be a minimal topological dynamical system. In this section
we prove several results that play a key role in the proof that $\NRP^{[d]}(X)$
are equivalence relations for $d\geq1$. These results are interesting
by their own right. The proofs use the theory of the Ellis semigroup
which we now recall.

\subsection{Ellis semigroup\label{sub:Ellis-semigroup}}

We very briefly review some theory related with the Ellis semigroup
(also known as the enveloping semigroup). A self-contained reference
is \cite[Chapter I]{glasner1976proximal}. We also recommend \cite[Appendix A]{SY12}.
\begin{defn}
The \textbf{Ellis semigroup} $E=E(G,X)$ of a t.d.s $(G,X)$ is the
closure of $G$ in the semigroup (with respect to composition) $X^{X}$
equipped with the product topology. The Ellis semigroup is compact
but in general not metrizable (see \cite{glasner2008metrizable}).
A dynamical morphism $\pi:(G,X)\to(G,Y)$ induces a surjective continuous
morphism of semigroups $\pi^{\bullet}:E(G,X)\to E(G,Y)$. Note that
for all $q\in E$, right multiplication in $E$ by $q$, $E\rightarrow E$,
$p\mapsto pq$ is continuous. An element $u\in E$ with $u^{2}=u$
is called an \textbf{idempotent}. A non-empty subset $I\subset E$
is a \textbf{left ideal} if $EI\subset I$. A \textbf{minimal} left
ideal is a left ideal that does not contain any proper left ideal
of $E$. Clearly any left ideal contains a minimal left ideal. An
idempotent contained in a minimal left ideal is called a \textbf{minimal
idempotent}. \end{defn}
\begin{prop}
\label{prop:Ellis fundamentals}Let $(G,X)$ be a t.d.s and $E$ its
Ellis semigroup. Suppose $L\subset E$ is a minimal left ideal and
and let $J(L)$ be the set of idempotents in $L$, then:\end{prop}
\begin{enumerate}
\item $J(L)\neq\emptyset$.
\item A point $x\in X$ is minimal if and only if there exists $u\in J(L)$
with $ux=x$.
\item Let $u$ be an idempotent in $E$. If $p\in Eu$, then $pu=p$.
\item Let $x\in X$ and $u\in E$ an idempotent, then $(x,ux)\in\PP(X)$.
In particular there is a minimal point which is proximal to $x$.
\item $L=\bigcup_{u\in J(L)}uE$ is a partition and every $uL$ is a group
with identity $u$.
\item $(G,E)$ is a t.d.s and $(G,L)$ is a minimal subsystem.
\item Let $x_{0}\in X$, then the map $(G,E)\to(G,X)$ given by $p\mapsto px_{0}$
is a dynamical morphism.\end{enumerate}
\begin{proof}
(1) \cite[Proposition I.2.3(1)]{glasner1976proximal}. (2) \cite[Proposition I.3.1(2)]{glasner1976proximal}.
(3) If $p\in Eu$ then $p=qu$ for some $q\in E$. Thus $pu=(qu)u=q(uu)=qu=p$.
(4) \cite[proof of Proposition I.3.2(2)]{glasner1976proximal}. (5)
\cite[Proposition I.2.3(3)]{glasner1976proximal}. (6) \cite[IV(3.7)(2) and IV(3.2)(2)]{dV93}.
(7) \cite[IV(3.7)(4)]{dV93}.
\end{proof}

\subsection{Induced projections\label{sub:Projections}}

Let $E=E(\Hcal^{[d]},C_{G}^{[d]}(X))$ be the enveloping semigroup
of $(\Hcal^{[d]},C_{G}^{[d]}(X))$. Let $\pi_{\ep}:C_{G}^{[d]}(X)\to X_{\ep}=X$
be the projection of $C_{G}^{[d]}(X)$ on the $\epsilon$-coordinate,
where $\ep\in\{0,1\}^{d}$. We consider the action of the group $\Hcal^{[d]}$
on the $\ep$-coordinate via the projection $\pi_{\ep}$, i.e, for
$\ep\in\{0,1\}^{d}$:
\[
\Hcal^{[d]}\times X_{\ep}\to X_{\ep},\quad gx\mapsto g_{\ep}x.
\]
With respect to this action of $\Hcal^{[d]}$ on $X_{\ep}=X$ the
map $\pi_{\ep}:(\Hcal^{[d]},C_{G}^{[d]}(X))\to(\Hcal^{[d]},X_{\ep})$
is a dynamical morphism. Let $\pi_{\ep}^{\bullet}:E(\Hcal^{[d]},C_{G}^{[d]}(X))~\to~E(\Hcal^{[d]},X_{\ep})$
be the corresponding homomorphisms of enveloping semigroups. Notice
that for the action of $\Hcal^{[d]}$ on $X_{\ep}$, $E(\Hcal^{[d]},X_{\ep})=E(G,X)$
as subsets of $X^{X}$. We claim that an element of $E(\Hcal^{[d]},C_{G}^{[d]}(X))$
is determined by its projections. Indeed as every element of $\Hcal^{[d]}$ acts on $C_{G}^{[d]}(X)$ coordinatewise,  this is also true for the closure of $\Hcal^{[d]}$ inside $(C_{G}^{[d]}(X))^{C_{G}^{[d]}(X)}$  therefore $E(\Hcal^{[d]},C_{G}^{[d]}(X))$ may be identified with a subset of $E(G,X)^{[d]}$ and moreover $\Hcal^{[d]}$ acts on $E(\Hcal^{[d]},C_{G}^{[d]}(X))$ coordinatewise.

Let $x_{0}\in X$. Consider the ceiling map $\pi_{c}:X^{[d]}\to X^{[d-1]}$ from Subsection \ref{sub:Discrete-cubes}. Let us denote its restriction to $C_{x_{0}}^{[d]}(X)$ also by $\pi_{c}$. We thus have a continuous map $\pi_{c}:C_{x_{0}}^{[d]}(X)\to X^{[d-1]}$. Similarly we have a ceiling map $\pi'_{c}:G^{[d]}\to G^{[d-1]}$. Let us denote its restriction to $\Hcal^{[d]}$  by $\phi_{c}$. We thus have a \textit{continuous group homomorphism} $\phi_{c}:\Hcal^{[d]}\to G^{[d-1]}$.
\begin{lem}
\label{lem:induced group projections}
$\phi_{c}(\Fcal^{[d]})= \Hcal^{[d-1]}$.
\end{lem}
\begin{proof}
Let $F_{i}=\{\o\in\{0,1\}^{d}|\ \o_{i}=1\}$, $1\leq i\leq d$, be the upper hyperfaces of $\{0,1\}^{d}$. Define the projection $\widehat{}:\{0,1\}^{d}\to \{0,1\}^{d-1}$ by $\widehat{x_1x_2\cdots x_d}\mapsto x_1\cdots x_{d-1}$. As $\Fcal^{[d]}$ is generated by $[G]_{F_{i}}$, $\phi_{c}(\Fcal^{[d]})$ is generated by $\phi_{c}([G]_{F_{i}})=[G]_{\widehat{F_{i}}}$ for $1\leq i\leq d$.
Note $\widehat{F_{1}},\ldots \widehat{F_{d-1}}$ are the upper hyperfaces of $\{0,1\}^{d-1}$, whereas $[G]_{\widehat{F_{d}}}=\{0,1\}^{d-1}$ and thus $[G]_{\widehat{F_{d}}}=\Del^{[d-1]}(G)$. Thus by Proposition \ref{prop:H=00003DFDiag}, $\phi_{c}(\Fcal^{[d]})=\Hcal^{[d-1]}$.
\end{proof}

\begin{lem}
\label{lem:induced projections range}
Let $(G,X)$ be a minimal t.d.s, then $\pi_{c}(C_{x_{0}}^{[d]}(X))\subset C_{G}^{[d-1]}(X)$.
\end{lem}
\begin{proof}
It follows from Proposition \ref{prop ap:cube inv} but let us give a direct proof. Clearly it is enough to prove $\pi_{c}(C_{G}^{[d]}(X))= C_{G}^{[d-1]}(X)$. By an argument similar to the proof of Lemma \ref{lem:induced group projections}, $\phi_{c}(\Hcal^{[d]})= \Hcal^{[d-1]}$. Using Proposition \ref{prop:face rep for Cd} twice we have as desired:

\begin{align*}
\pi_{c}(C_{G}^{[d-1]}(X))=\pi_{c}(\overline{\{gx_{0}^{[d]}|g\in\Hcal^{[d]}\}})=\overline{\{\phi_{c}(g)x_{0}^{[d-1]}|g\in\Hcal^{[d]}\}}  \\
= \overline{\{gx_{0}^{[d-1]}|g\in\Hcal^{[d-1]}\}}=C_{G}^{[d-1]}(X)
\end{align*}
\end{proof}

Let $(H,X)$ and $(H',X')$ be t.d.s where possibly $H\neq H'$. Let us say that a pair of maps $(f,\phi)$ is a dynamical morphism between
$(H,X)$ and $(H',X')$ if $f:X\to X'$ is a continuous map,  $\phi:H\to H'$ is a continuous group homomorphism and for all $x\in X$ and $g\in H$, $f(gx)=\phi(g)f(x)$.
The next simple lemma will be used  in the next subsection.
\begin{lem}
\label{lem:induced projections} Let $(G,X)$ be a minimal t.d.s, then the pair $(\pi_{c},\phi_{c})$ is a dynamical morphism between $(\Fcal^{[d]},C_{x_{0}}^{[d]}(X))$ and $(\Hcal^{[d-1]},C_{G}^{[d-1]}(X))$.
\end{lem}

\begin{proof}
By Lemmas \ref{lem:induced group projections} and \ref{lem:induced projections range} respectively $\pi_{c}:C_{x_{0}}^{[d]}(X)\to C_{G}^{[d-1]}(X)$ is a continuous map and  $\phi_{c}:\Fcal^{[d]}\to \Hcal^{[d-1]}$ is a continuous group homomorphism. Finally it is easy to see for all $b\in C_{x_{0}}^{[d]}(X)$ and $g\in \Fcal^{[d]}$, $\pi_{c}(gb)=\phi_{c}(g)\pi_{c}(b)$.
\end{proof}

\subsection{Minimal actions}

In \cite[Lemma 4.1]{HKM10} it was proven that $(\Hcal^{[d]},C_{\mathbb{Z}}^{[d]}(X))$
is minimal for $(\mathbb{{Z}},X)$ minimal and distal t.d.s. It was
also mentioned that Glasner had shown (unpublished) that one can remove
the distality assumption. Here we show that the same statement holds
for a general group action. We note that the essential feature of $\Hcal^{[d]}$ which is used in the proof is that it contains the diagonal, i.e. $\Del^{[d]}\subset\Hcal^{[d]}$.
\begin{prop}
\label{prop:minimality} Let $(G,X)$ be a minimal t.d.s, then the
t.d.s $(\Hcal^{[d]},C_{G}^{[d]}(X))$ is minimal. \end{prop}
\begin{proof}
Let $x\in X$ and let $u$ be a minimal idempotent in $E(G,X)$ with
$ux=x$ (Proposition \ref{prop:Ellis fundamentals}(2)). Then $\tilde{u}\triangleq u^{[d]}\in E(\Hcal^{[d]},C_{G}^{[d]}(X))$
and $\tilde{u}$ is an idempotent. Our goal is to show that $\tilde{u}$
is a minimal idempotent of $E(\Hcal^{[d]},C_{G}^{[d]}(X))$. Given
that this is true, as $x^{[d]}=\tilde{u}x^{[d]}$, by Proposition
\ref{prop:Ellis fundamentals}(2), $C_{G}^{[d]}(X)$ which is the
orbit closure of $x^{[d]}$, is $\Hcal^{[d]}$-minimal as desired.
Choose $v$ a minimal idempotent in the closed left ideal $E(\Hcal^{[d]},C_{G}^{[d]}(X))\tilde{u}$
(Proposition \ref{prop:Ellis fundamentals}(1)). As $\tilde{u}$ is
an idempotent, $v\tilde{u}=v$ (Proposition \ref{prop:Ellis fundamentals}(3)).
We will show that $\tilde{u}v=\tilde{u},$ which implies that the
idempotent $\tilde{u}$ belongs to the minimal left ideal $E(\Hcal^{[d]},C_{G}^{[d]}(X))v$
and thus is minimal. Set, for $\epsilon\in\{0,1\}^{d},\ v_{\epsilon}=\pi_{\epsilon}^{\bullet}v$
($\pi_{\epsilon}^{\bullet}$ is defined in Subsection \ref{sub:Projections}).
Note that, as an element of $E(\Hcal^{[d]},C_{G}^{[d]}(X))$ is determined
by its projections, it suffices to show that for each $\epsilon,\ uv_{\epsilon}=u$.
Since for each $\epsilon$ the map $\pi_{\epsilon}^{\bullet}$ is
a semigroup homomorphism, we have that $v_{\epsilon}u=v_{\epsilon}$
as $v\tilde{u}=v$, and
%at
$v_{\epsilon}v_{\epsilon}=v_{\epsilon}$
as $vv=v$. In particular we deduce that $v_{\epsilon}$ is an idempotent
belonging to the minimal left ideal $E(G,X_{\epsilon})u=E(G,X)u$
and thus $u\in E(G,X_{\epsilon})v_{\epsilon}$ by Proposition \ref{prop:Ellis fundamentals}(6). By Proposition \ref{prop:Ellis fundamentals}(3),
this implies that $uv_{\epsilon}=u$ and it follows that indeed $\tilde{u}v=\tilde{u}$.%
\end{proof}

Define:

\[
Y_{x}^{[d]}(X)=\ol{\Fcal^{[d]}(x^{[d]})}\subset C_{x}^{[d]}(X)
\]

\[
Y_{x*}^{[d]}(X)=\pi_{*}(Y_{x}^{[d]})\subset C_{x*}^{[d]}(X)
\]

In \cite[Proposition 4.2]{HKM10} it was proven that $Y_{x}^{[d]}(X)=C_{x}^{[d]}(X)$ and it clearly follows that
for $(\mathbb{{Z}},X)$ minimal and distal,
for each $x\in X$, the system
$(Y_{x}^{[d]}(X),Y_{x}^{[d]}(X))$ is minimal.
In \cite[Theorem 3.1]{SY12}
 it was shown, using
 the structure theory of minimal systems,
 that for abelian group actions, for
 each $x\in X$, the system $(\Fcal^{[d]},Y_{x}^{[d]}(X))$ is minimal.
 Here we show that the same statement holds for a general
group action using only enveloping semigroup arguments. We start by an auxiliary lemma:
\begin{lem}
\label{lem:idempotent}Let $u\in E(G,X)$ be an idempotent. Then $u_{*}^{[d]}\in E(\Fcal^{[d]},Y_{x*}^{[d]}(X))$. \end{lem}
\begin{proof}
Enumerate the upper hyperfaces of $\{0,1\}^{d}$ by $F_{1},F_{2},\ldots,F_{d}$.
Let $t_{\alpha}\in G$ be a net in $G$ such that $t_{\alpha}\rightarrow_{\alpha}u$
in $E(G,X)$. As $[t_{i}]_{F_{1}}\in\Fcal^{[d]}$, we have $[u]_{F_{1}}\in E(\Fcal^{[d]},Y_{x}^{[d]}(X))$.
As $[t_{i}]_{F_{2}}[u]_{F_{1}}\in E(\Fcal^{[d]},Y_{x}^{[d]}(X))$,
$E\rightarrow E$, $p\mapsto pq$ is continuous and $u^{2}=u$, we
have $[u]_{F_{1}\cup F_{2}}\in E(\Fcal^{[d]},Y_{x}^{[d]}(X))$. We
now continue similarly for $F_{3},F_{4},\ldots,F_{d}$.
\end{proof}

\begin{prop}\label{prop:Yx is minimal}
Let $(G,X)$ be a minimal t.d.s, then for each $x\in X$, the t.d.s
$(\Fcal^{[d]},Y_{x}^{[d]}(X))$, and hence also $(\Fcal^{[d]},Y_{x*}^{[d]}(X))$,
are minimal. \end{prop}
\begin{proof}
The proof of the minimality of the t.d.s $(\Fcal^{[d]},Y_{x*}^{[d]}(X))$
is almost verbatim the same as in the proof of Proposition \ref{prop:minimality},
except that here the claim that for $u$ a minimal idempotent in $E(G,X)$,
the map $\tilde{u}=u_{*}^{[d]}$ is in $E(\Fcal^{[d]},Y_{x*}^{[d]}(X))$,
is not that evident. However, as $u$ is an idempotent this fact follows
from Lemma \ref{lem:idempotent}.
\end{proof}

 In \cite[Theorem 3.1]{SY12} it was proven that for each
$x\in X,$ $(\Fcal^{[d]},Y_{x}^{[d]}(X))$ is the unique minimal subsystem
of $(\Fcal^{[d]},C_{x}^{[d]}(X))$ for $(G,X)$ minimal t.d.s with
$G$ abelian. Here we show that the same statement holds for a general
group action. We start by proving a lemma is a generalization of the ``useful lemma''
\cite[Lemma 5.1]{SY12}. The proof follows closely the original proof
with one exception: the use of the pure ceiling-mixed decomposition
(Subsection \ref{sub:HK decomposition}).
\begin{lem}
\label{lem:useful} Let $(G,X)$ be a minimal t.d.s and $d\geq1$.
If $(x^{[d-1]},w)\in C_{x}^{[d]}(X)$ for some $x\in X$ and $w\in C_{G}^{[d-1]}(X)$
and $(x^{[d-1]},w)$ is an $\Fcal^{[d]}$-minimal point, then $(x^{[d-1]},w)\in Y_{x}^{[d]}(X)$. \end{lem}
\begin{proof}
We will show that there exists a minimal left ideal $L\subset E(\Fcal^{[d]},C_{x}^{[d]}(X))$
and an idempotent $v\in L$ such that $(\pi_{f}^{\bullet}(v)x^{[d-1]},w)\in Y_x^{[d]}$.
Assume this is true. Since, by assumption, $(x^{[d-1]},w)$ is $\Fcal^{[d]}$-minimal,
there is some minimal idempotent $u\in J(L)$ such that $u(x^{[d-1]},w)=(\pi_{f}^{\bullet}(u)x^{[d-1]},\pi_{c}^{\bullet}(u)w)=(x^{[d-1]},w)$
(Proposition \ref{prop:Ellis fundamentals}(2)). Since $u,v\in L$
are minimal idempotents in the same minimal left ideal $L$, we have $u\in Ev$ and this implies
$uv=u$ (Proposition \ref{prop:Ellis fundamentals}(3)). Thus $u(\pi_{f}^{\bullet}(v)x^{[d-1]},w)=(\pi_{f}^{\bullet}(u)\pi_{f}^{\bullet}(v)x^{[d-1]},\pi_{c}^{\bullet}(u)w)=(\pi_{f}^{\bullet}(uv)x^{[d-1]},w)=(\pi_{f}^{\bullet}(u)x^{[d-1]},w)=(x^{[d-1]},w)$
which implies $(x^{[d-1]},w)\in Y_x^{[d]}$ as desired.

To construct $L$ and $v$ notice that since $(x^{[d-1]},w)\in C_{G}^{[d]}(X)$
and $(\Hcal^{[d]},C_{G}^{[d]}(X))$ is a minimal t.d.s by Proposition
\ref{prop:minimality}, $(x^{[d-1]},w)$ is in the $\Hcal^{[d]}$-orbit
closure of $x^{[d]}$, i.e. by Proposition \ref{prop:H=00003Dpure ceiling mixed}
there are sequences $(s_{n}\times s_{n})\in\Hcal^{[d]}$ and $(\id^{[d-1]}\times h_n)\in\Fcal^{[d]}$
where $s_{n},h_n\in G^{[d-1]}$ such that:
\[
(\id^{[d-1]}\times h_n)(s_{n}\times s_{n})(x^{[d-1]},x^{[d-1]})=(s_{n}x^{[d-1]},h_ns_{n}x^{[d-1]})\to_{n}(x^{[d-1]},w).
\]
letting $a_{n}\triangleq s_{n}x^{[d-1]}=\pi_{f}(s_{n}x^{[d-1]},h_ns_{n}x^{[d-1]})\in C_{G}^{[d-1]}(X)$,
we have:
\begin{equation}
(\id^{[d-1]}\times h_n)(a_{n},a_{n})\to(x^{[d-1]},w)\label{eq:a,a limit}
\end{equation}

Fix a minimal left ideal $L$ of $E(\Fcal^{[d]},C_{x}^{[d]}(X))$.
By Proposition \ref{prop:Ellis fundamentals}(6) $(\Fcal^{[d]},L)$
is a minimal subsystem of $E(\Fcal^{[d]},C_{x}^{[d]}(X))$. Thus by
Lemma \ref{lem:induced projections} $\pi_{c}(L)\subset E(\Hcal^{[d-1]},C_{G}^{[d-1]}(X))$
is a minimal $\Hcal^{[d-1]}$-subsystem.\\ As$(\Hcal^{[d-1]},C_{G}^{[d-1]}(X))$
is minimal by Proposition \ref{prop:minimality} it follows from Proposition
\ref{prop:Ellis fundamentals}(7) that $\pi_{c}(L)x^{[d-1]}=C_{G}^{[d-1]}(X)$.
Thus there exist $p_{n}\in L$ such that $a_{n}=\pi_{c}^{\bullet}(p_{n})x^{[d-1]}$.
Let $p\in L$ be an accumulation point of $\{p_{n}\}$. As by (\ref{eq:a,a limit}),
$\pi_{c}^{\bullet}(p_{n})x^{[d-1]}=a_{n}\to_{n}x^{[d-1]}$ we must
have
\begin{equation}
\pi_{c}^{\bullet}(p)x^{[d-1]}=x^{[d-1]}.\label{q2}
\end{equation}
If $\pi_{f}^{\bullet}(p)x^{[d-1]}=\pi_{f}^{\bullet}(v)x^{[d-1]}$
for some idempotent $v\in L$ then $b_{n}\triangleq(\id^{[d-1]}\times h_n)p_{n}x^{[d]}\in Y_{x}^{[d]}(X)$
and $b_{n}=(\pi_{f}^{\bullet}(p_{n})x^{[d-1]},h_n\pi_{c}^{\bullet}(p_{n})x^{[d-1]}\to_{n}(\pi_{f}^{\bullet}(v)x^{[d-1]},w)$
as desired. However as this does not necessarily hold, the idea is
to find an element $q\in E(\Fcal^{[d]},C_{x}^{[d]}(X))$ and $v\in J(L)$
so that $v=pq$ and $\pi_{c}^{\bullet}(pq)x^{[d-1]}=x^{[d-1]}$. Defining
$x_{n}=p_{n}qx^{[d]}\in Y_{x}^{[d]}(X)$, one has $(\id^{[d-1]}\times h_n)x_{n}\to_{n}(\pi_{f}^{\bullet}(v)x^{[d-1]},w)$
as desired (see details below).

Indeed since $L$ is a minimal left ideal and $p\in L$, by Proposition
\ref{prop:Ellis fundamentals}(5) there exists a minimal idempotent
$v\in J(L)$ such that $vp=p$. Thus:
\begin{equation}
\pi_{c}^{\bullet}(v)x^{[d-1]}=\pi_{c}^{\bullet}(v)\pi_{c}^{\bullet}(p)x^{[d-1]}=\pi_{c}^{\bullet}(vp)x^{[d-1]}=\pi_{c}^{\bullet}(p)x^{[d-1]}=x^{[d-1]}\label{eq:v in F}
\end{equation}
By Proposition \ref{prop:Ellis fundamentals}(5) $vL$ is a group.
One verifies easily the following is a subgroup:
\[
S=\{a\in vL:\pi_{c}^{\bullet}(a)x^{[d-1]}=x^{[d-1]}\}
\]
By (\ref{q2}), we have that $vp=p\in S$. Let S so that $pq=v$.
Thus $\pi_{c}^{\bullet}(pq)x^{[d-1]}=x^{[d-1]}$. Denote $x_{n}=p_{n}qx^{[d]}\in Y_{x}^{[d]}(X)$.
Note $\pi_{f}(x_{n})=\pi_{f}^{\bullet}(p_{n}q)x^{[d-1]}\to_{n}\pi_{f}^{\bullet}(pq)x^{[d-1]}=\pi_{f}^{\bullet}(v)x^{[d-1]}$
and $\pi_{c}(x_{n})=\pi_{c}^{\bullet}(p_{n}q)x^{[d-1]}=a_{n}$. As
$\id^{[d-1]}\times h_n\in\Fcal^{[d]}$, $(\id^{[d-1]}\times h_n)x_{n}\in Y_{x}^{[d]}(X)$.
By (\ref{eq:a,a limit}), $(\id^{[d-1]}\times h_n)x_{n}\to_{n}(\pi_{f}^{\bullet}(v)x^{[d-1]},w)$
and thus we conclude as desired $(\pi_{f}^{\bullet}(v)x^{[d-1]},w)\in Y_{x}^{[d]}(X)$.
\end{proof}

With the above preparation we are ready to show:

\begin{thm}
\label{thm:unique minimal}Let $(G,X)$ be a minimal topological dynamical
system and $d\geq1$, then for each $x\in X,$ $(\Fcal^{[d]},Y_{x}^{[d]}(X))$
is the unique minimal subsystem of $(\Fcal^{[d]},C_{x}^{[d]}(X))$.
Hence also $(\Fcal^{[d]},Y_{x*}^{[d]}(X))$ is the unique minimal
subsystem of the t.d.s. $(\Fcal^{[d]},C_{x*}^{[d]}(X))$.\end{thm}
\begin{proof}
For $d=1$ the claim is obvious as $Y_{x}^{[1]}(X)=C_{x}^{[1]}(X)$.
We assume by induction that the assertion holds for every $1\le j\le d-1$
and given $x\in X$, consider a minimal subsystem $Y$ of the t.d.s
$(\Fcal^{[d]},C_{x}^{[d]}(X))$. Let $\pi_{f}$ be the floor projection
(see Subsection \ref{sub:Projections}). We observe that $Y_{1}=\pi_{f}(Y)$
is a minimal subsystem of the t.d.s $(\Fcal^{[d-1]},C_{x}^{[d-1]}(X))$
and therefore, by the induction hypothesis $Y_{1}=Y_{x}^{[d-1]}(X)=\ol{\Fcal^{[d-1]}x^{[d-1]}}$.
But then for some $w\in C_{G}^{[d-1]}(X)$ we have $(x^{[d-1]},w)\in Y$.
Therefore the claim is reduced to the ``useful lemma\textquotedbl{}
\cite[Lemma 5.1]{SY12} which we reproduce as Lemma \ref{lem:useful}
in the sequel.\end{proof}
\begin{cor}\label{cor:approaching constant}
 Let $(G,X)$ be a minimal t.d.s
and $d\geq1$. If $c\in C_{x}^{[d]}(X)$ then $x^{[d]}\in\overline{\Fcal^{[d]}c}$.\end{cor}
\begin{proof}
Assume not, then there is more than one $\Fcal^{[d]}$-minimal
subsystem in $C_{x}^{[d]}(X)$ contradicting Theorem \ref{thm:unique minimal}. \end{proof}
\begin{cor}
\label{cor:xyy is minimal} Let $(G,X)$ be a minimal t.d.s and $d\geq1$.
Assume $\urcorner^{[d]}(x,y)\in C_{x}^{[d]}(X)$ then $\urcorner^{[d]}(x,y)\in Y_x^{[d]}(X)$.
In particular $\urcorner^{[d]}(x,y)$ is a $\Fcal^{[d]}$-minimal point.\end{cor}
\begin{proof}
Note $(\Fcal^{[d]},\overline{\Fcal^{[d]}y^{[d]}})\to(\Fcal^{[d]},\overline{\Fcal^{[d]}\urcorner^{[d]}(x,y)})$
is an $\Fcal^{[d]}$-isomorphism.
Thus $(\Fcal^{[d]},\overline{\Fcal^{[d]}\urcorner^{[d]}(x,y)})\subset C_{x}^{[d]}(X)$
is $\Fcal^{[d]}$-minimal. By Theorem \ref{thm:unique minimal},
\allowbreak
$\overline{\Fcal^{[d]}\urcorner^{[d]}(x,y)}~=~Y_x^{[d]}(X)$
and the result follows.\end{proof}
\begin{cor}
\label{cor:xxyx is minimal} Let $(G,X)$ be a minimal t.d.s and $d\geq2$.
Assume $\urcorner^{[d]}(x,y)\in C_{x}^{[d]}(X)$ then $(x^{[d-1]},\urcorner^{[d-1]}(y,x))\in Y_x^{[d]}(X)$.\end{cor}
\begin{proof}
By Corollary \ref{cor:approaching constant} there is a sequence $f_k\in \Fcal^{[d-1]}$ such
that $f_{k}\urcorner^{[d-1]}(x,y)\to x^{[d-1]}$. Conclude $$(f_{k}\times f_{k})\urcorner^{[d]}(x,y)=(f_{k}\times f_{k})(\urcorner^{[d-1]}(x,y),y^{[d-1]})\to(x^{[d-1]},\urcorner^{[d-1]}(y,x)).$$
By Corollary \ref{cor:xyy is minimal}, $\urcorner^{[d]}(x,y)\in Y_x^{[d]}(X)$.
As $f_{k}\times f_{k}\in\Fcal^{[d]}$ (see Subsection \ref{sub:doubling})
the result follows.
\end{proof}
In \cite[Proposition 4.2]{HKM10} it was proven that $Y_{x}^{[d]}(X)=C_{x}^{[d]}(X)$
for $(\mathbb{{Z}},X)$ minimal and distal t.d.s for each $x\in X$.
Here we show that the same statement holds for a general group action.
\begin{thm}
Let $(G,X)$ be a minimal distal topological dynamical system and
$d\geq1$, then for each $x\in X$, $Y_{x}^{[d]}(X)=C_{x}^{[d]}(X)$
and hence $Y_{x*}^{[d]}(X)=C_{x*}^{[d]}(X)$.\end{thm}
\begin{proof}
By \cite[Theorem 5.6]{A} $(G^{[d]},X^{[d]})$ is a distal t.d.s.
This immediately implies that $(\Fcal^{[d]},C_{x}^{[d]}(X))$ is a
distal t.d.s. By \cite[Corollary 5.4(iii)]{A}, a distal system is
\textit{semisimple}, i.e decomposes into a disjoint union of minimal
subsystems. By Theorem \ref{thm:unique minimal}, $(\Fcal^{[d]},C_{x}^{[d]}(X))$
has a unique minimal subsystem $(\Fcal^{[d]},Y_{x}^{[d]}(X))$. We
thus conclude $Y_{x}^{[d]}(X)=C_{x}^{[d]}(X)$.
\end{proof}

\begin{rem}
There are non-distal minimal t.d.s for which $Y_{x*}^{[d]}(X)\neq C_{x*}^{[d]}(X)$. See \cite[Example 3.6]{tu2013dynamical}.
\end{rem}

Let $Z$ be a compact metric space and let $2^{Z}$ denote the \textit{hyperspace}
consisting of the closed non-empty subsets of $Z$ equipped with the
(compact metric) Vietoris topology (\cite[p. 124]{akin2010general}).
A function $X\rightarrow2^{Z}$ is called \textit{lower-semi-continuous}
at $x\in X$ if for every open set $O\subset Z$ such that $f(x)\cap O\neq\emptyset$, we have
that $\{y\in X|\:f(y)\cap O\neq\emptyset\}$ is a neighborhood
of $x$. A function $X\rightarrow2^{Z}$ is called \textit{upper-semi-continuous}
at $x\in X$ if for every open set $O\subset Z$ such that $f(x)\subset O$, we have that $\{y\in X|\:f(y)\subset O\}$ is a neighborhood of $x$
(\cite[Proposition 7.11]{akin2010general}). A function $X\rightarrow2^{Z}$
is continuous at $x\in X$  with respect to the Vietoris topology iff it is both upper
and lower semi-continuous at $x\in X$ (\cite[Lemma 7.5]{akin2010general}).

The following theorem is new even for $G=\mathbb{{Z}}$.
\begin{thm}\label{thm:genericity}
Let $(G,X)$ be a minimal topological dynamical system where $X$ is metrizable, then for a dense $G_{\del}$ subset
$X_{0}\subset X$ one has $Y_{x}^{[d]}(X)=C_{x}^{[d]}(X)$.\end{thm}
\begin{proof}
Consider $\Phi:X\to2^{X^{[d]}}$ given by $x\mapsto Y_{x}^{[d]}(X)$.
It is easy to check that this map is lower-semi-continuous. By \cite[Theorem 7.19]{akin2010general}
the set of continuity points of $\Phi$ is a dense $G_{\del}$ subset
$X_{0}\subset X$. Since by Proposition \ref{prop:face rep for Cd}
the set $\Fcal^{[d]}\Del^{[d]}(X)$ is dense in $C_{G}^{[d]}(X)$,
it follows that at each point of $X_{0}$ we must have $Y_{x}^{[d]}(X)=C_{x}^{[d]}(X)$.
Indeed let $x_{0}\in X_{0}$ and assume $Y_{x_{0}}^{[d]}(X)\neq C_{x_{0}}^{[d]}(X)$.
Let $U$ be an open set in $C_{G}^{[d]}(X)$ so that $Y_{x_{0}}^{[d]}(X)\subset C_{x_{0}}^{[d]}(X)\cap U\neq C_{x_{0}}^{[d]}(X)$.
As $\Phi$ is
upper-semi-continuous at $x_{0}$
the set $\{x\in X|\:Y_{x}^{[d]}(X) \subset U\}$
is a neighborhood of $x_{0}$ and it follows $\Fcal^{[d]}\Del^{[d]}$
is not dense in $C_{x_{0}}^{[d]}(X)$.
\end{proof}

%\subsection{Auxiliary lemmas}

\section{$\NRP^{[d]}(X)$ is an equivalence relation for minimal actions\label{sec:RPd is an equivalence relation}}
In this section we prove the main theorem of the article, Theorem
\ref{thm:main}:
\begin{proof}
Clearly $\NRP^{[d]}(X)$ $(d\geq1)$ is a closed $G$-invariant and
reflexive. To prove symmetry assume $(y,x)\in\NRP^{[d]}(X)$, i.e.
$\llcorner^{[d+1]}(y,x)\in C_{G}^{[d+1]}(X)$. Permuting coordinates
(see Proposition \ref{prop:dynamical cubespaces are cubespaces}) we
have $\urcorner^{[d+1]}(x,y)\in C_{G}^{[d+1]}(X)$. Projecting we have
$\urcorner^{[d]}(x,y)\in C_{G}^{[d]}(X)$. By Corollary \ref{cor:approaching constant}
there is a sequence $f_{k}\in\Fcal^{[d]}$ so that $f_{k}\urcorner^{[d]}(x,y)\to x^{[d]}$.
Thus $(f_{k}\times f_{k})\urcorner^{[d+1]}(x,y)=(f_{k}\times f_{k})(\urcorner^{[d]}(x,y),y^{[d]})\to(x^{[d]},\urcorner^{[d]}(y,x))$
(see Subsection \ref{sub:doubling}). Permuting coordinates again
we have $\llcorner^{[d+1]}(x,y)\in C_{G}^{[d+1]}(X)$ as desired.

To prove transitivity assume $(x,y),(y,z)\in\NRP^{[d]}(X)$. By symmetry
$(z,y)\in\NRP^{[d]}(X)$. Permuting coordinates we have $\urcorner^{[d+1]}(y,x),\urcorner^{[d+1]}(y,z)\in C_{G}^{[d+1]}(X)$.
By Corollary \ref{cor:xyy is minimal} $\urcorner^{[d+1]}(y,x),\urcorner^{[d+1]}(y,z)$
induce minimal $\Fcal^{[d+1]}$-subsystems of $C_{y}^{[d+1]}(X)$
and thus by Theorem \ref{thm:unique minimal} $\urcorner^{[d+1]}(y,z)\in\overline{\Fcal^{[d+1]}\urcorner^{[d+1]}(y,x)}$.
As $(\Fcal^{[d+1]},\overline{\Fcal^{[d+1]}\urcorner^{[d+1]}(y,x)})\to Y_{x}^{[d+1]}(X)$
given by $w\mapsto(x,w_{*})$ is an $\Fcal^{[d+1]}$-isomorphism it
follows that $\urcorner^{[d+1]}(x,z)\in Y_{x}^{[d+1]}(X)$. Permuting
coordinates we have $\llcorner^{[d+1]}(z,x)\in C_{G}^{[d+1]}(X)$ .
Thus $(z,x)\in\NRP^{[d]}(X)$. By symmetry $(x,z)\in\NRP^{[d]}(X)$
as desired.
\end{proof}

\section{Lifting $\NRP^{[d]}(X)$ from factors to extensions\label{sec:Lifting--from}}

Let $(G,X)$ be a minimal t.d.s. In Lemma \ref{lem:elementary} we
note that if $\pi:(G,X)\to(G,Y)$ is a dynamical morphism then $\pi\times\pi(\NRP^{[d]}(X))\subset\NRP^{[d]}(Y)$.
In \cite[Theorem 6.4]{SY12} it is proven that for $G$ abelian equality
holds, i.e, $\pi\times\pi(\NRP^{[d]}(X))=\NRP^{[d]}(Y)$. We next show that the same is true for general minimal group actions. Our proof follows
the framework of the proof of \cite[Theorem 6.4]{SY12}.
\begin{thm}
\label{thm:lifting}Let $(G,X)$ be a minimal topological dynamical
system. If $\pi:(G,X)\to(G,Y)$ is a dynamical morphism then $\pi\times\pi(\NRP^{[d]}(X))=\NRP^{[d]}(Y)$.\end{thm}
\begin{proof}
Let $(y_{1},y_{2})\in\NRP^{[d]}(Y)$. Our goal is to find $(x_{1},x_{2})\in\NRP^{[d]}(X)$
such that $\pi(x_{1})=y_{1}$ and $\pi(x_{2})=y_{2}$. This will be
referred to as in the sequel as lifting $(y_{1},y_{2})$. By Proposition
\ref{prop:Ellis fundamentals}(4) there is a minimal point $(y_{1}',y_{2}')\in\overline{\Or((y_{1},y_{2}),G)}$
such that $(y_{1}',y_{2}')$ is proximal to $(y_{1},y_{2})$. Note
$(y_{1}',y_{2}')\in\NRP^{[d]}(Y)$ as $\NRP^{[d]}(Y)$ is $G$-invariant
and closed. Since $(y_{1},y_{1}'),(y_{2},y_{2}')\in P(Y)$, then by
\cite[Lemma 6.3]{SY12} there are $x_{1},x_{2}\in X$ such that $\pi\times\pi(x_{1},x_{2})=(y_{1},y_{2})$
and $(x_{1}',x_{1})$, $(x_{2}',x_{2})$ $\in$ $\PP(X)$. By Lemma
\ref{lem:elementary}(1) $(x_{1}',x_{1})$, $(x_{2}',x_{2})$ $\in$
$\NRP^{[d]}(X)$. Assume we have proven one can lift $(y_{1}',y_{2}')$,
i.e., there is $(x_{1}',x_{2}')\in\NRP^{[d]}(X)$ with $\pi\times\pi(x_{1}',x_{2}')=(y_{1}',y_{2}')$.
By the transitivity of $\NRP^{[d]}(X)$ (Theorem \ref{thm:main}),
$(x_{1},x_{1}'),(x_{1}',x_{2}'),(x_{2}',x_{2})\in\NRP^{[d]}(X)$ imply
$(x_{1},x_{2})\in\NRP^{[d]}(X)$. Hence we can assume without loss of generality that $(y_{1},y_{2})$ is a minimal point of $(Y\times Y,G)$.

Let $q_{1}\in\pi^{-1}(y_{1})$. We will find $q_{2}\in\pi^{-1}(y_{2})$
such that $(q_{1},q_{2})\in\NRP^{[d]}(X)$ and such that some $(x_{1},x_{2})$
in the orbit closure of $(q_{1},q_{2})$ lifts $(y_{1},y_{2})$. As
an intermediary step we construct cubes in $C_{q_{1}}^{[d+1]}(X)$
with an increasing number of vertices whose value is $q_{1}$.

As $(y_{1},y_{2})\in\NRP^{[d]}(Y)$, by Corollary \ref{cor:xxyx is minimal}
there is a sequence $f_{k}\in\Fcal^{d+1},$ $f_{k}y_{1}^{[d+1]}\to(y_{1}^{[d-1]},\urcorner^{[d-1]}(y_{2},y_{1}))$.
Let $c\in C_{q_{1}}^{[d+1]}(X)$ be an accumulation point of the sequence
$f_{k}q_{1}^{[d+1]}$. Note $\pi(c)=(y_{1}^{[d-1]},\urcorner^{[d-1]}(y_{2},y_{1}))$.
Let $F_{i}=\{\o\in\{0,1\}^{d+1}|\ \o_{i}=0\}$ be an enumeration of lower hyperfaces of $\{0,1\}^{d+1}$. Inductively we will construct elements $c_{d+1},c_{d},\ldots,c_{1}\in C_{q_{1}}^{[d+1]}(X)$ and
$z_{d+1}=y_{2},z_{d},\ldots,z_{1}\in Y$ such that for $i=d+1,d,\ldots,1$:
\begin{enumerate}
\item $c_{i}(\omega)=q_{1}$ for $\omega\in F_{d+1}\cup F_{d}\cup\cdots\cup F_{i}$
\item $\pi(c_{i})(0\cdots0\stackrel{i}{1}\cdots1)=z_{i}$
\item $\pi(c_{i})(\omega)=y_{1}$ for all $\omega\neq(0\cdots0\stackrel{i}{1}\cdots1)$
\item $(z_{i},y_{1})\in\overline{\Or((y_{1},z_{i+1}),G)}$ (only for $i\leq d)$
\end{enumerate}
Assume this has been achieved. Let us consider the element $c_{1}$.
As $F_{d+1}\cup F_{d}\cup\cdots\cup F_{1}=\{0,1\}_{*}^{d+1}$, we
have $c_{1}=\llcorner_{d+1}(q_{1},q_{2})$ for some $q_{2}\in X$.
Thus $(q_{1},q_{2})\in\NRP^{[d]}(X)$. By (2) $\pi(q_{2})=z_{1}$.
By ($4)$
\[
(z_{1},y_{1})\in\overline{\Or((y_{1},z_{2}),G)}\subset\overline{\Or((z_{3},y_{1}),G)}\subset\overline{\Or((y_{1},z_{4}),G)}\subset\cdots
\]
Thus $(z_{1},y_{1})\in\overline{\Or((z_{d+1},y_{1}),G)}$ or $(z_{1},y_{1})\in\overline{\Or((y_{1},z_{d+1}),G)}$.
Assume without loss of generality the first case. As $z_{d+1}=y_{2}$
and $(y_{1},y_{2})$ is a minimal point of $(Y\times Y,G)$, $(y_{1},y_{2})\in\overline{\Or((y_{1},z_{1}),G)}$.
Let $g_{k}\in G$ so that $(g_{k}y_{1},g_{k}z_{1})\to(y_{1},y_{2})$.
Assume without loss of generality $(g_{k}q_{1},g_{k}q_{2})\to(x_{1},x_{2})$.
As $\NRP^{[d]}(X)$ is $G$-invariant and closed, $(x_{1},x_{2})\in\NRP^{[d]}(X)$.
Moreover we have, as desired:

\[
\pi\times\pi(x_{1},x_{2})=\lim_{k}{}(g_{k}\pi(q_{1}),g_{k}\pi(q_{2}))=\lim_{k}(g_{k}y_{1},g_{k}z_{1})=(y_{1},y_{2})
\]
We now return to the inductive construction of $c_{d+1},c_{d},\ldots,c_{1}\in C_{q_{1}}^{[d+1]}(X)$.
By Corollary \ref{cor:approaching constant}, there is a sequence
$f_{k}\in\Fcal^{[d]}$ such that $f_{k}c_{|F_{d+1}}\to q_{1}^{[d-1]}$.
Let $c_{d+1}\in C_{q_{1}}^{[d+1]}(X)$ be an accumulation point of
the sequence $f_{k}\times f_{k}c$. Thus $c_{d+1}(\omega)=q_{1}$
for $\omega\in F_{d+1}$ and property (1) holds for $i=d+1$. Combining
$f_{k}\pi(c)_{|F_{d+1}}\to y_{1}^{[d-1]}$ with $\pi(c)=(y_{1}^{[d-1]},\urcorner^{[d-1]}(y_{2},y_{1}))$
we have $\pi(c_{d+1})=\lim_{k}\pi(f_{k}\times f_{k}(y_{1}^{[d-1]},\urcorner^{[d-1]}(y_{2},y_{1})))=(y_{1}^{[d-1]},\urcorner^{[d-1]}(y_{2},y_{1}))$
which implies property (2). Thus denoting $z_{d+1}=y_{2}$ we have
$\pi(c_{d+1})(0\cdots01)=z_{d+1}$ which is property (2) for $i=d+1$.

\begin{figure}[h]
\centering{}\begin{tikzpicture}[thick,scale=2]
\node (2A1) at (0, 0){$q_1$};
\node (2A2) at (0, 1){};
  \node (2A3) at (1, 1){};
  \node (2A4) at (1, 0){$q_1$};
 \node (2B1) at (0.5, 0.5){$q_1$};
   \node (2B2) at (0.5, 1.5){};
   \node (2B3) at (1.5, 1.5){};
 \node (2B4) at (1.5, 0.5){$q_1$};
    \draw (2A1) -- (2A2);
  \draw (2A2) -- (2A3);
\draw (2A3) -- (2A4);
 \draw (2A4) -- (2A1);
    \draw[dashed] (2A1) -- (2B1);
  \draw[dashed] (2B1) -- (2B2);
 \draw[very thick] (2A2) -- (2B2);
  \draw[very thick] (2B2) -- (2B3);
  \draw[very thick] (2A3) -- (2B3);
  \draw[very thick] (2A4) -- (2B4);
  \draw[very thick] (2B4) -- (2B3);
\draw[dashed] (2B1) -- (2B4);

    \node (A1) at (2, 0){$y_1$};
   \node (A2) at (2, 1){$z_3$};
\node (A3) at (3, 1){$y_1$};
   \node (A4) at (3, 0){$y_1$};
   \node (B1) at (2.5, 0.5){$y_1$};
 \node (B2) at (2.5, 1.5){$y_1$};
 \node (B3) at (3.5, 1.5){$y_1$};
   \node (B4) at (3.5, 0.5){$y_1$};
    \draw (A1) -- (A2);
  \draw (A2) -- (A3);
 \draw (A3) -- (A4);
  \draw (A4) -- (A1);
    \draw[dashed] (A1) -- (B1);
    \draw[dashed] (B1) -- (B2);
   \draw[very thick] (A2) -- (B2);
  \draw[very thick] (B2) -- (B3);
    \draw[very thick] (A3) -- (B3);
   \draw[very thick] (A4) -- (B4);
 \draw[very thick] (B4) -- (B3);
  \draw[dashed] (B1) -- (B4);
\end{tikzpicture} \caption{$c_{3}$ and $\pi(c_{3})$ for $d=2$.}
\end{figure}

\noindent Assume we have already constructed $c_{i+1}\in C_{q_{1}}^{[d+1]}(X)$
and $z_{i+1}\in Y$. By Corollary \ref{cor:approaching constant},
there is a sequence $f_{k}\in\Fcal^{[d]}$ such that $f_{k}c_{i+1|F_{i}}\to q_{1}^{[d-1]}$.
Let $c_{i}\in C_{q_{1}}^{[d+1]}(X)$ be an accumulation point of the
sequence $D_{i}(f_{k})c_{i+1}$ (for the notation $D_{i}(\cdot)$
see Subsection \ref{sub:doubling}). Clearly $c_{i|F_{i}}=q_{1}^{[d-1]}$.
In order to establish property (1), we have to show in addition that for
$\omega\in(F_{d+1}\cup F_{d}\cup\cdots\cup F_{i+1})\setminus F_{i}$
it holds that
$c_{i}(\omega)=q_{1}$. Define: $\phi_{i}:\{0,1\}^{d+1}\to F_{i}$
to be the projection on $F_{i}$, i.e., $\phi_{i}(\omega_{1}\omega_{2}\cdots\omega_{i}\cdots\omega_{d+1})=(\omega_{1}\omega_{2}\cdots\stackrel{i}{0}\cdots\omega_{d+1})$.
Fix $\omega\in F_{j}\setminus F_{i}$ for $j>i.$ As $\omega,\phi_{i}(\omega)\in F_{j}$,
$c_{i+1}(\omega)=c_{i+1}(\phi_{i}(\omega))=q_{1}.$ By the definition
of doubling, the same is true for $D_{i}(f_{k})$, i.e., $D_{i}(f_{k})(\phi_{i}(\omega))=D_{i}(f_{k})(\omega)$
and thus we conclude $c_{i}(\omega)=c_{i}(\phi_{i}(\omega))=q_{1}$
as desired, where the last equality follows from $\phi_{i}(\omega)\in F_{i}.$
Denote $\pi(c_{i})(0\cdots0\stackrel{i}{1}\cdots1)=z_{i}.$ We now
establish property (3). If $\omega\in F_{i}$ then as $c_{i}(\omega)=q_{1}$
it follows $\pi(c_{i}(\omega))=y_{1}$. Thus we only need to treat
the case $\omega\in F_{i}^{c}\setminus\{(0\cdots0\stackrel{i}{1}~\cdots1)\}$.
By the inductive construction $\pi(c_{i+1})(\omega)=y_{1}$ for all
$\omega\neq(0\cdots0\stackrel{i+1}{1}~\cdots1)$. Note that $\phi_{i}(\omega)=(0\cdots0\stackrel{i+1}{1}\cdots1)$
implies $\omega=(0\cdots0\stackrel{i}{1}\cdots1)$. Thus as $(0\cdots0\stackrel{i+1}{1}\cdots1)\in F_{i}$,
we conclude that for $\omega\in F_{i}^{c}\setminus\{(0\cdots0\stackrel{i}{1}\cdots1)\}$,
$\pi(c_{i+1}(\omega))=\pi(c_{i+1}(\phi_{i}(\omega)))=y_{1}$. By the
definition of doubling, $\pi(c_{i}(\omega))=\pi(c_{i}(\phi_{i}(\omega)))=\pi(q_{1})=y_{1}$
as desired. From property (3) we have\\ $\pi(c_{i}(0\cdots0\stackrel{i+1}{1}~\cdots1))~=~y_{1}$
which implies for $g_{k}=D_{i}(f_{k})(0\cdots0\stackrel{i}{1}\cdots1)=D_{i}(f_{k})(0\cdots0\stackrel{i+1}{1}~\cdots1)$ that

\[
y_{1}=\lim_{k}g_{k}\pi(c_{i+1}(0\cdots0\stackrel{i+1}{1}\cdots1))=\lim_{k}g_{k}z_{i+1}.
\]

\noindent %
\begin{comment}
where the last equality follows from the definition of doubling.
\end{comment}
Similarly as $\pi(c_{i+1}(0\cdots0\stackrel{i}{1}\cdots1))=y_{1}$,
$z_{i}=\lim_{k}g_{k}y_{1}$. We thus have $(z_{i},y_{1})=\lim_{k}g_{k}\times g_{k}(y_{1},z_{i+1})$
which is property (4).

\begin{figure}[h]
\centering{}\begin{tikzpicture}[thick,scale=2]
\node (2A1) at (0, 0){$q_1$};
\node (2A2) at (0, 1){$q_1$};
  \node (2A3) at (1, 1){$q_1$};
  \node (2A4) at (1, 0){$q_1$};
 \node (2B1) at (0.5, 0.5){$q_1$};
   \node (2B2) at (0.5, 1.5){};
   \node (2B3) at (1.5, 1.5){};
 \node (2B4) at (1.5, 0.5){$q_1$};
    \draw (2A1) -- (2A2);
  \draw (2A2) -- (2A3);
\draw (2A3) -- (2A4);
 \draw (2A4) -- (2A1);
    \draw[dashed] (2A1) -- (2B1);
  \draw[dashed] (2B1) -- (2B2);
 \draw[very thick] (2A2) -- (2B2);
  \draw[very thick] (2B2) -- (2B3);
  \draw[very thick] (2A3) -- (2B3);
  \draw[very thick] (2A4) -- (2B4);
  \draw[very thick] (2B4) -- (2B3);
\draw[dashed] (2B1) -- (2B4);

    \node (A1) at (2, 0){$y_1$};
   \node (A2) at (2, 1){$y_1$};
\node (A3) at (3, 1){$y_1$};
   \node (A4) at (3, 0){$y_1$};
   \node (B1) at (2.5, 0.5){$y_1$};
 \node (B2) at (2.5, 1.5){$z_2$};
 \node (B3) at (3.5, 1.5){$y_1$};
   \node (B4) at (3.5, 0.5){$y_1$};
    \draw (A1) -- (A2);
  \draw (A2) -- (A3);
 \draw (A3) -- (A4);
  \draw (A4) -- (A1);
    \draw[dashed] (A1) -- (B1);
    \draw[dashed] (B1) -- (B2);
   \draw[very thick] (A2) -- (B2);
  \draw[very thick] (B2) -- (B3);
    \draw[very thick] (A3) -- (B3);
   \draw[very thick] (A4) -- (B4);
 \draw[very thick] (B4) -- (B3);
  \draw[dashed] (B1) -- (B4);
\end{tikzpicture} \caption{$c_{2}$ and $\pi(c_{2})$ for $d=2$.}
\end{figure}

\end{proof}

\section{Systems of order $d$ \label{sec:systems-of-order}}

\subsection{Overview}

\label{subsec:Overview}

In this section we investigate the structure of minimal systems whose
regionally proximal relation of order $d$ is trivial.
\begin{defn}
Let $d\geq1$. A t.d.s $(G,X)$ is called a \textbf{system of order
$d$ }if $d$ is the minimal integer such that $\NRP^{[d]}(X)=\Delta$.

The fundamental example of systems of order \textbf{$d$ }is given
by nilsystems:
\end{defn}

\begin{defn}
Let $d\ge1$ be an integer and assume that $L$ is a nilpotent Lie\footnote{A \textit{Lie group} is a second countable topological group $G$ that has a differentiable structure such that the map $G^{2}\to G:(g,h)\mapsto gh^{-1}$
is differentiable. Note we do not assume that Lie groups are connected.
In particular countable discrete groups are Lie.} group of nilpotency class $d$ and $\Gamma\subset L$ a discrete,
cocompact subgroup of $L$. Denote $X=L/\Gamma$. Notice that $L$ acts naturally
on $X$ by left translations: $l\Gamma\rightarrow gl\Gamma$ for $g\in L$.
Let $G$ be a topological group and let $\phi:G\rightarrow L$ be
a continuous homomorphism, then the induced action $(G,X)$ is called
a \textbf{nilsystem of order $d$}.\end{defn}
\begin{thm}
\label{thm: ex_nilsystems}Let $(G,X)$ be a minimal nilsystem of order $d$ where
$G$ is an arbitrary topological group, then it is a system of order
at most \textbf{$d$}.\end{thm}
\begin{proof}
See Example \ref{ex:nrp}.
\end{proof}

A natural question which arises is if one can characterize systems
of order $d$ in terms of nilsystems. We will return to this question
in Subsection \ref{sub:Strong-structure-theorem}. In the meantime
we will opt for a more abstract treatment. The next corollary provides
a canonical way to generate systems of order at most \textbf{$d$}.
\begin{cor}
\label{cor:canonical factor is d system}Let $(G,X)$ be a minimal
t.d.s, then $\NRP^{[d]}(X/\NRP^{[d]}(X))=\Delta$, i.e. $(G,X/\NRP^{[d]}(X))$
is a system of order at most $d$.\end{cor}
\begin{proof}
Let $Y=X/\NRP^{[d]}(X)$ and $\pi:X\to Y$ the associated factor map.
By Theorem \ref{thm:lifting}, $\NRP^{[d]}(Y)=\pi\times\pi(\NRP^{[d]}(X))=\Delta$.
\end{proof}
The next theorem shows that dividing out by the regionally proximal
relation results with the \textit{maximal} factor which is a system
of order at most $d$:
\begin{thm}
\label{thm:maximal d factor}Let $d\geq1$ and let $(G,X)$ be a minimal
topological dynamical system, then $\pi_{d}:(G,X)\to(G,X/\NRP^{[d]}(X))$
is the maximal factor of order at most $d$ of $(G,X)$. That is,
if $\phi:(G,X)\to(G,Y)$ is a factor map where $(G,Y)$ is a system
of order at most $d$, then there exists a factor map $\psi:(G,X/\NRP^{[d]}(X))\to(G,Y)$
such that $\phi=\psi\circ\pi_{d}$.\end{thm}
\begin{proof}
By Corollary \ref{cor:canonical factor is d system} $(G,X/\NRP^{[d]}(X))$
is a system of order at most $d$. It is enough to show that $(x,y)\in\NRP^{[d]}(X)$,
implies $\phi(x)=\phi(y)$. Indeed by Lemma \ref{lem:elementary}(5)
$(x,y)\in\NRP^{[d]}(X)$ implies $(\phi(x),\phi(y))\in\NRP^{[d]}(Y)=\Delta$
and thus $\phi(x)=\phi(y)$.\end{proof}
\begin{rem}
\label{rem:Systems-of-order d are distal}Systems of finite order
are distal.\end{rem}
\begin{proof}
By Lemma \ref{lem:elementary}(1), $\PP(X)=\Delta.$
\end{proof}
We now move on to more advanced structure theorems for systems of
order $d$. The key tool is the theory of nilspaces introduced by
Antol\'\i n Camarena and Szegedy. We review this theory in Subsection \ref{sub:Nilspaces}
and in Subsection \ref{sub:Distal-systems-are} we prove that minimal
systems of finite order are nilspaces. This allows us to adapt the
so-called weak structure theorem of Antol\'\i n Camarena and Szegedy to the dynamical
context in Subsection \ref{sub:Weak-structure-theorem}. In Subsection
\ref{sub:Strong-structure-theorem} we quote the stronger Gutman-Manners-Varj{\'u} structure theorem for systems of finite order which hold under some
restrictions on the acting group.

\subsection{Nilspaces\label{sub:Nilspaces}}

A map $f=(f_{1},\ldots,f_{k}):\{0,1\}^{d}\to\{0,1\}^{k}$ is called
a \textbf{morphism of discrete cubes} if each coordinate function
$f_{j}(\o_{1},\ldots,\o_{d})$ is either identically  $0$,
identically $1$, or it equals either $\o_{i}$
or $\overline{\o_{i}}=1-\o_{i}$ for some $1\le i=i(j)\le d$.

In Subsection \ref{sub:Dynamical-cubespaces}
we introduced dynamical cubespaces. A (general) \textbf{cubespace} is a pair
$(X,{C^{\bullet}})$ consisting of a compact metric space $X$  together with a collection of closed subsets
$C^{[d]}(X)\subseteq X^{[d]}$,  for each integer $d\ge0$, called \textbf{cubes}, so that for any morphism of discrete cubes
$f:\{0,1\}^{d}\to\{0,1\}^{k}$ and any $c\in C^{[k]}(X)$, we have
$c\circ f\in C^{[d]}(X)$. We refer to this property as \textbf{cube
invariance}. When no confusion arises we denote the cubespace simply
by $X$. It is not hard to verify that dynamical cubespaces are cubespaces
(See Proposition \ref{prop:dynamical cubespaces are cubespaces}). We say that a cubespace $(X,{C_{G}^{\bullet}})$ is \textbf{ergodic}, if $C^{1}(X)=X^{[1]} = X \times X$, that is to say, if any pair of elements forms
a $1$-cube.

Let $X$ be a cubespace and let $f:X\to\{0,1\}_{*}^{d}$ be a map.
We call $f$ a d-\textbf{corner} if $f|_{\{\o\in\{0,1\}^{d+1}|\ \o_{i}=0\}}$
is a $(d-1)$-cube for all $1\le i\le d$. We say that the cubspace$(X,{C^{\bullet}})$
has \textbf{$d$-completion} if for any $d$-corner $f$, there is a cube $c\in C^{[d]}(X)$ such that $c|_{\{0,1\}_{*}^{d}}=f$.
We say that $(X,{C_{G}^{\bullet}})$ is \textbf{fibrant} if it has
$d$-completion for all $d\ge1$.

\begin{example}\label{ex:weakly mixing}
Recall that $(G,X)$ is called \textbf{transitive}, if for every pair of
non-empty open subsets $U$ and $V$, there is $g\in G$ such that $U\cap gV\not=\emptyset$; is called \textbf{weakly mixing} if the diagonal action $(\Del^{2}(G),X)$ is transitive; and is called \textbf{transitive of all orders} if the diagonal action $(\Del^{n}(G),X)$ is transitive for all $n\in \mathbb{N}$. An example of a t.d.s which is fibrant is given by a minimal system which is transitive of all orders\footnote{For $G$ abelian transitivity of all orders is equivalent to weak mixing (\cite[Theorem 1.11]{G03}). For $G$ non-abelian the conditions are not equivalent (\cite[p. 277]{weiss2000survey}). If $(G,X)$ is minimal and admits an invariant measure with full support with respect to which it is measurably weakly mixing then it is transitive of all orders (\cite[Theorem 6.12]{akin2008topological})}. See Proposition \ref{prop:weakly mixing}.
\end{example}

 We say that $(X,{C^{\bullet}})$
has \textbf{$d$-uniqueness} if the following holds: whenever $c,c'\in C^{[d]}(X)$
and $c(\omega)=c'(\omega)$ for all $\omega\in\{0,1\}_{*}^{d}$ then
$c=c'$.

We say that a cubespace $(X,{C_{G}^{\bullet}})$ is a \textbf{nilspace
of order $d$ }if it is fibrant and $d\ge0$ is the smallest integer
such that $X$ has $(d+1)$-uniqueness.

Let $X$ be a cubespace and let $\sim$ be a closed equivalence relation
on $X$. One endows $X/{\sim}$ by a cubespace structure by declaring
a configuration $c\in(X/{\sim})^{[d]}$ a cube if and only if there
is a cube $c'\in C^{[d]}(X)$ such that $\pi(c')=c$. It is clear
that $X/{\sim}$ is indeed a cubespace.

Let $X$ be a fibrant cubespace. Define $x\sim_{d}y$ if and only
if there are two cubes $c_{1},c_{2}\in C^{[d+1]}(X)$ such that $c_{1}(\o)=c_{2}(\o)$
for $\o\neq\vec{1}$ and $c_{1}(\vec{1})=x$ and $c_{2}(\vec{1})=y$.
Denote $\pi_{d}:X\rightarrow X/{\sim}_{d}$. By \cite[Proposition 6.3]{GMVI}
(following \cite[Section 2.4]{CS12} and \cite[Section 3.3]{HK08})
$\sim_{d}$ is an equivalence relation and $\pi_{d}(X)$ is a nilspace.
We call $X/{\sim}_{d}$ the \textbf{$d$-th canonical factor} of $X$.
The following remark is trivial:
\begin{rem}
\label{rem:uniqueness}Let $d\geq0$. A cubespace $X$ has $(d+1)-$uniqueness
iff $\sim_{d}=\Delta$.
\end{rem}
The relation between successive canonical factors is elucidated by
the so-called weak structure theorem proven by Antol\'\i n Camarena and Szegedy
in \cite[Theorem 1]{CS12}. A detailed exposition is given in \cite[Chapters 6 \& 7]{GMVI}.
We quote a partial version of the theorem:
\begin{thm}
\label{th:weak-structure} Let $X$ be an ergodic nilspace of order
at most $d$. Then there is an additive compact abelian group $A_{d}$
acting continuously and freely on $X$ such that the orbits of $A_{d}$
coincide with the fibres of $\pi_{d-1}:X\to X/{\sim}_{d-1}$.
\end{thm}
Iterating the theorem we see that a nilspace of finite order can be
represented by a finite tower of compact abelian group extensions:

\begin{equation}
X\to\pi_{d-1}(X)\to\pi_{d-2}(X)\to\ldots\to\pi_{0}(X)=\bullet\label{eq:non dyn tower}
\end{equation}

In Subsection \ref{sub:Weak-structure-theorem} we will adapt this
theorem to the dynamical context.

\subsection{Minimal distal systems are fibrant\label{sub:Distal-systems-are}}
\begin{thm}
\label{thm:completion for distal} Let $(G,X)$ be a minimal distal
topological dynamical system, then the cubespace $(X,{C_{G}^{\bullet}})$
is ergodic and fibrant.
\end{thm}
The fact that $(X,{C_{G}^{\bullet}})$ is ergodic follows trivially
from minimality of $(G,X)$. The proof that $(X,{C_{G}^{\bullet}})$
is fibrant splits into a number of lemmas, which are based on \cite[Section 4.2]{HKM10}.

In this subsection we will identify $\{0,1\}^{d}$ with the collection
of all subsets of $\{1,\ldots,d\}$ and write $\o'\subseteq\o$ for
$\o',\o\in\{0,1\}^{d}$ if $\o'(i)\le\o(i)$ for all $i$.

Let $V\subseteq\{0,1\}^{d}$ be a \textbf{downwards-closed} subset,
i.e.~if $\o\in V$ and $\o'\subseteq\o$ then $\o'\in V$. Denote
by $\Hom(V,X)$ the set of maps $\a:V\to X$ such that for all $\o\in V$,
$\a|_{\{\o'|\ \o'\subseteq\o\}}$ is a cube of $X$.
\begin{lem}
\label{lem:distality of action on Hom} Let $(G,X)$ be a distal t.d.s
and $V\subseteq\{0,1\}^{d}$ a downwards-closed subset. Then $(\Hcal^{[d]},\Hom(V,X))$
equipped with the coordinate-wise action is a distal system.\end{lem}
\begin{proof}
By \cite[Chapter 5, Theorem 6]{A} $(G^{[d]},X^{[d]})$ is a distal
system. As $\Hom(V,X)\subset X^{[d]}$ this immediately implies that
$(\Hcal^{[d]},\Hom(V,X))$ is a distal system.
\end{proof}
In particular, for $\a_{1},\a_{2}\in\Hom(V,X)$, we have $\a_{1}\in\overline{\Or(\a_{2},\Hcal^{[d]})}$
if and only if $\a_{2}\in\overline{\Or(\a_{1},\Hcal^{[d]})}$.

Let $V\subseteq\{0,1\}^{d}$ be a downwards-closed subset. We say
that $\Hom(V,X)$ has the \textbf{extension property} if for every
$\a\in\Hom(V,X)$, there exists $c\in C_{G}^{[d]}(X)$ so that $c|_{V}=\a$.
Note that a cubespace $(X,C^{\bullet}(X))$ has $d$-completion if
and only if $\Hom(\{0,1\}_{*}^{d},X)$ has the extension property.
Therefore Theorem \ref{thm:completion for distal} follows from the
next lemma.
\begin{lem}
\label{lem:distal has extension property} Let $(G,X)$ be a minimal
distal t.d.s and let $V\subseteq\{0,1\}^{d}$ be a downwards-closed
subset, then $\Hom(V,X)$ has the extension property.\end{lem}
\begin{proof}
We prove the lemma by a double induction; first we induct on $d$,
then on the cardinality of $V$. If $d=1$, the claim is clear. We
assume that the claim holds for downward-closed subsets in $\{0,1\}^{d-1}$
and prove it for downward-closed subsets in $\{0,1\}^{d}$.

Let $V$ be a downward-closed subset in $\{0,1\}^{d}$. If $V=\{\vec{0}\}$,
the result is clear. Assume $|V|\geq2$ and $V\neq\{0,1\}^{d}$. Let
$\vec{1}\neq\wt\o\in V$ be a maximal element in $V$ and denote $W=V\setminus\{\wt\o\}$.
(Note that $W\neq\varnothing$.)

Let $\a\in\Hom(V,X)$. We first consider the special case that $\a|_{W}\equiv x$
for some $x\in X$. We show that $\a$ can be extended to a cube.
Let $1\le i\le d$ be such that $\wt\o_{i}=0$ and define $F=\{\o\in\{0,1\}^{d}|\ \o_{i}=0\}$
and $E=\{\o\in\{0,1\}^{d}|\ \o=1\}$ .

By the inductive assumption, $\a|_{V\cap F}$ can be extended to a
map $c_{1}:F\to X$ that is a cube. Let $c_{2}=D_{i}(c_{1}).$ By
Subsection \ref{sub:doubling} $c_{2}\in C_{G}^{[d]}(X)$.

We show that $c_{2}$ is an extension of $\a$. This is clearly true
on $F\cap V$. Let $\o\in E\cap V$. As $\tau(\o)\subseteq\o$ and
$V$ is a downward-closed subset, we must have $\tau(\o)\in V$. Moreover,
$\o\neq\wt\o$ as $\wt\o$ is maximal in $V$. Thus
\[
c_{2}(\o)=c_{1}\circ\tau(\o)=\a\circ\tau(\o)=x=\a(\o).
\]

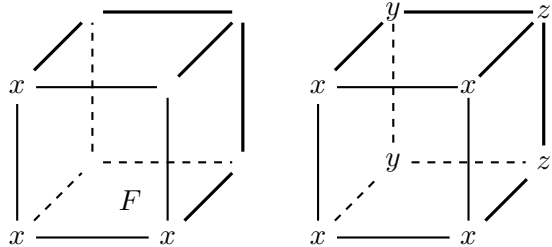
\begin{figure}[h]
\centering{}\begin{tikzpicture}[thick,scale=2]
    \node (2A1) at (0, 0){$x$};
   \node (2A2) at (0, 1){$x$};
 \node (2A3) at (1, 1){$$};
\node (2A4) at (1, 0){$x$};
   \node (2B1) at (0.5, 0.5){$$};
 \node (2B2) at (0.5, 1.5){$$};
\node (2B3) at (1.5, 1.5){$$};
  \node (2B4) at (1.5, 0.5){$$};
 \node (F) at (0.75, 0.25){$F$};
 \draw (2A1) -- (2A2);
  \draw (2A2) -- (2A3);
  \draw (2A3) -- (2A4);
  \draw (2A4) -- (2A1);
 \draw[dashed] (2A1) -- (2B1);
   \draw[dashed] (2B1) -- (2B2);
 \draw[very thick] (2A2) -- (2B2);
 \draw[very thick] (2B2) -- (2B3);
  \draw[very thick] (2A3) -- (2B3);
  \draw[very thick] (2A4) -- (2B4);
  \draw[very thick] (2B4) -- (2B3);
 \draw[dashed] (2B1) -- (2B4);
 \node (2B42) at (1.5, 0.5){$$};
   \node (A1) at (2, 0){$x$};
   \node (A2) at (2, 1){$x$};
  \node (A3) at (3, 1){$$};
 \node (A4) at (3, 0){$x$};
   \node (B1) at (2.5, 0.5){$y$};
 \node (B2) at (2.5, 1.5){$$};
  \node (B3) at (3.5, 1.5){$$};
  \node (B4) at (3.5, 0.5){$z$};
   \draw (A1) -- (A2);
 \draw (A2) -- (A3);
\draw (A3) -- (A4);
 \draw (A4) -- (A1);
    \draw[dashed] (A1) -- (B1);
  \draw[dashed] (B1) -- (B2);
 \draw[very thick] (A2) -- (B2);
 \draw[very thick] (B2) -- (B3);
 \draw[very thick] (A3) -- (B3);
 \draw[very thick] (A4) -- (B4);
  \draw[very thick] (B4) -- (B3);
   \draw[dashed] (B1) -- (B4);
    \node (A32) at (3, 1){$x$};
   \node (B22) at (2.5, 1.5){$y$};
   \node (B32) at (3.5, 1.5){$z$};
   \end{tikzpicture} \caption{An example of $\alpha$ and $c_{2}$.}
\end{figure}

\noindent We now return to the general case. By the inductive assumption,
$\a|_{W}$ can be extended to a cube $c_{1}\in C_{G}^{[d]}(X)$. By
Proposition \ref{prop:minimality}, $(\Hcal^{[d]},C_{G}^{[d]}(X))$
is minimal. Therefore, we can find a sequence $h_{i}\in\Hcal^{[d]}$
such that $\lim_{i}h_{i}c_{1}=x^{[d]}$.

Let $\a'=\lim_{i}h_{i}.\a$ (we can assume without loss of generality
that the limit exists). By Lemma \ref{lem:distality of action on Hom},
$\Hom(V,X)$ is invariant under the action of $\Hcal^{[d]}$, and
therefore $\a'\in\Hom(V,X)$. As we have $\a'|_{W}\equiv x$ in addition,
we may conclude by the previous case that $\a'$ can be extended to
a cube $c_{2}$.

Using Lemma \ref{lem:distality of action on Hom}, we can find a sequence
$g_{i}\in\Hcal^{[d]}$ such that $\a=\lim_{i}g_{i}.\a'$. We conclude
that $\lim_{i}g_{i}.c_{2}$ is an extension of $\a$ (again we can assume
without loss of generality that the limit exists).
\end{proof}
We are now ready to prove that minimal systems of finite order are
nilspaces. The key observation is that the canonical equivalence relation
$\sim_{s}$ has the following alternative definition:
\begin{prop}
\label{prop:alter-canonical} Let $X$ be a fibrant cubespace and
$d\geq1$ , then $x\sim_{d}y$ if and only if $\llcorner_{d+1}(x,y)$
is a cube.
\end{prop}
\begin{proof}
This is proven in \cite[Lemma 6.6]{GMVI}
(see also \cite[Lemma 2.3]{CS12} and \cite[Proposition 3]{HK08}).
\end{proof}
We now prove:
\begin{thm}
\label{thm:distal of finite order is nilspace} Let $d\geq1$ and
let $(G,X)$ be a minimal topological dynamical system, then $(G,X)$ is a system of order at
most $d$ iff the cubespace $(X,{C_{G}^{\bullet}})$ is an ergodic
nilspace of order at most $d$.\end{thm}
\begin{proof}
Assume that $(G,X)$ is a system of order at
most $d$. By Remark \ref{rem:Systems-of-order d are distal}, $(G,X)$ is distal.
In view of Theorem \ref{thm:completion for distal} one has only to
establish that $(X,{C_{G}^{\bullet}})$ has $(d+1)-$uniqueness. By
Remark \ref{rem:uniqueness} this is equivalent to $\sim_{d}=\Delta$.
By Proposition \ref{prop:alter-canonical}, $\NRP^{[d]}(X)=\sim_{d}$.
As $\NRP^{[d]}(X)=\Delta$, the result follows. Conversely if $(X,{C_{G}^{\bullet}})$ is a nilspace of order at most $d$ then $\NRP^{[d]}(X)=\sim_{d}=\Delta$.
\end{proof}

In \cite{TaoBlog15} Tao asks for "an interpretation of the regionally proximal relation in the nilspace language." We believe the following theorem answers his question:

\begin{thm}
\label{thm:answer_to_Tao} Let $d\geq1$ and
let $(G,X)$ be a minimal topological dynamical system, then $\pi_{d}:(G,X)\to(G,X/\NRP^{[d]}(X))$
is the maximal factor which is an ergodic
nilspace of order at most $d$.\end{thm}
\begin{proof}
Follows from Theorem \ref{thm:distal of finite order is nilspace} and Theorem \ref{thm:maximal d factor}.
\end{proof}

\subsection{Weak structure theorem for minimal systems of finite order.\label{sub:Weak-structure-theorem}}

In this subsection we adapt the so-called weak structure theorem of
Antol\'\i n Camarena and Szegedy (see Theorem \ref{th:weak-structure}) to the
dynamical context. First we introduce the appropriate terminology:
\begin{defn}
(See \cite[p.15]{G03} and \cite[V(4.1)]{dV93}) A dynamical morphism
$f:(G,X)\to(G,X)$ is called an \textbf{automorphism }if $f$ is bijective.
The group of automorphisms\textbf{ }equipped with the uniform topology
is denoted by $\Aut(G,X)$. A dynamical morphism $\pi:(G,X)\to(G,Y)$
is called a \textbf{principal abelian group extension} if there exists
a compact abelian group $K\subset\Aut(G,X)$ such that for all $x,y\in X$,
$\pi(x)=\pi(y)$ iff there exists a unique $k\in K$ such that $kx=y$.
If $Y=\bullet$, then $(G,X)$ is called an \textbf{abelian group
t.d.s. }It is not hard to see that $(G,X)$ is a minimal abelian group
t.d.s if and only if $X$ is a compact abelian group and there exists
a continuous group homomorphism $\phi:G\to X$ with $\overline{\phi(G)}=X$
such that $G$ acts through $gx=\phi(g)+x$. %

\end{defn}
Our main result in this subsection is:
\begin{thm}
\label{thm:weak structure dynamical}Let $d\geq1$ and let $(G,X)$
be a minimal topological dynamical system of order at most $d$, then
the following is a sequence of principal abelian group extensions:

\begin{equation}\label{eq:tower}
(G,X)\to(G,X/\NRP^{[d-1]}(X))\to\cdots\to(G,X/\NRP^{[1]}(X))\to\bullet
\end{equation}

\end{thm}
In particular $(G,X/\NRP^{[1]}(X))$ is an abelian group t.d.s.
\begin{proof}
As $\NRP^{[d]}(X)=\sim_{d}$, the tower structure (\ref{eq:tower}) is
a direct consequence of (\ref{eq:non dyn tower}), however one has
to show that the successive maps in (\ref{eq:tower}) are\textbf{ }principal
abelian group extensions. As $\NRP^{[d]}(X)$ is a $G$-equivariant closed
equivalence relation the maps are dynamical morphisms. By Theorem
\ref{th:weak-structure} there is an additive compact abelian group
$K_{d}$ acting continuously and freely on $X$ such that the orbits
of $K_{d}$ coincide with the fibres of $\pi_{d-1}:(G,X)\to(G,X/\NRP^{[d-1]}(X))$.
We will show $K_{d}\subset\Aut(G,X)$. From \cite[p.45]{GMVI}:

\[
K_{d}=\NRP^{[d-1]}(X)/\sim
\]
where $(x,x')\sim(y,y')$ if and only if $(\llcorner_{d}(x,x'),\llcorner_{d}(y,y'))\in C_{G}^{[d+1]}(X)$.
Denote the equivalence classes by $[x,x']$. These classes corresponds
to the elements of $K_{d}$. Fix $x\in X$ , $a\in K_{d},$ $g\in G$.
We have to show that the equality $a(gx)=g(ax)$ holds.
Denote $x'=ax$. By
definition (\cite[p.47]{GMVI}), $(x,x')\in\NRP^{[d-1]}(X)$ and $a=[x,x']$.
We conclude $\llcorner_{d}(x,x')\in C_{G}^{[d]}(X)$. By doubling
$(\llcorner_{d}(x,x'),\llcorner_{d}(x,x'))\in C_{G}^{[d+1]}(X)$ (see
Subsection \ref{sub:doubling}) and this implies $(\llcorner_{d}(x,x'),\llcorner_{d}(gx,gx'))\in C_{G}^{[d+1]}(X)$
by Equation (\ref{eq:def of Cn H}) in Subsection \ref{sub:Dynamical-cubespaces}.
Thus $(x,x')\sim(gx,gx')$ which implies $a=[gx,gx']=[gx,g(ax)]$,
i.e $a(gx)=g(ax)$ as desired.
\end{proof}

\subsection{Strong structure theorem for some systems of finite order.\label{sub:Strong-structure-theorem}}

In Theorem \ref{thm: ex_nilsystems} we saw that minimal nilsystems are systems
of finite order. It is not hard to see that an inverse limit of
nilsystems of uniformly bounded order is a system of finite order. It turns out that
under some restrictions on the acting group one can prove
that
these are
the only possible examples.%
\begin{comment}
Indeed for a considerable family of actions one can prove that a system
of finite order is a \textit{pronilsystem}, i.e an inverse limit
of nilsystems.
\end{comment}
We qoute \cite[Theorem 1.29]{GMVIII}:
\begin{thm}\label{thm:stg str}
\label{thm:strong structure}Let $d\geq1$ and let $(G,X)$ be a minimal
topological dynamical system of order at most $d$, where $G$ has
a dense subgroup generated by a compact set and where $X$ is metrizable.
%011116
%{\color{blue}
Then $(G,X)$ is a pronilsystem of order at most $d$.
%, that is,
%}
\end{thm}

%011116
We recall that the system $(G,X)$ is a pronilsystem of order at most $d$
when:
\begin{itemize}
\item There exists a sequence of nilpotent Lie groups $G^{(n)}$ of nilpotency
class at most $d$;
\item for each $n$, there is a continuous homomorphism $\alpha_{n}:G\to G^{(n)}$;
\item for each $n$, there is a discrete co-compact subgroup $\Gamma^{(n)}\subseteq G^{(n)}$;
and
\item for each $n>m$, there is a continuous homomorphism $\psi_{m,n}:G^{(n)}\to G^{(m)}$,
\end{itemize}
such that
\begin{itemize}
\item $\psi_{m,n}(\Gamma_{n})\subseteq\Gamma_{m}$,
\item $\alpha_{m}=\psi_{m,n}\circ\alpha_{n}$,
\item and $(G,X)$ is isomorphic as a topological dynamical system to the
inverse limit of the nilsystems $(G,G^{(n)}/\Gamma^{(n)})$ given by the inverse
system of maps induced by $\psi_{m,n}$, where $G$ acts on $G^{(n)}/\Gamma^{(n)}$
via $\alpha_{n}$.
\end{itemize}

\begin{rem}
A minimal t.d.s  isomorphic to a tower of principal abelian group extensions as in (\ref{eq:tower}) is not necessarily of finite order. Consider the famous Furstenberg counerexample (\cite[end of Subsection 3.1]{furstenberg1961strict}, see also \cite[Chapter 5.5]{P81}) of a homeomorphism of the torus $\mathbb{S}^2$ of the form $T:(x,y)\mapsto (x+\alpha, y+\phi(x))$ which is minimal distal but not uniquely ergodic. Denote $\pi:\mathbb{S}^2\mapsto \mathbb{S}$ by $(x,y)\mapsto x $. Then $\pi$ realizes $(\mathbb{S}^2,T)$ as a circle extension of the maximal equicontinuous factor which is also a circle. Note however that $(\mathbb{S}^2,T)$ is not a finite order system. Indeed by a classical Theorem of Green (\cite{auslander1963flows}, see also \cite{parry1970dynamical}) a minimal $\mathbb{Z}$-nilsystem is uniquely ergodic. Thus by the above Theorem  \ref{thm:stg str}  a finite order $\mathbb{Z}$-system is uniquely ergodic.
\end{rem}

\section{A different generalization of $\RP^{[d]}(X)$}

\subsection{The relation between $\NRP^{[1]}$ and $\Q(X)$\label{sec:The-relation-between}}

Recall the definitions of $\Q(X)$ and $\Qeq(X)$ from Subsection
\ref{sub:Proximality-and-its} and the introduction. In this section we investigate the
relation between $\Q(X)$ and $\NRP^{[1]}(X)$ and characterize $(G,X/\NRP^{[1]}(X))$.
We start with a simple proposition.
\begin{prop}
\label{prop:Q subset RP1} Let $(G,X)$ be a minimal t.d.s. If $(x,y)\in\Q(X)$
then $(x,y)\in\NRP^{[1]}(X)$. Thus $\PP(X)\subset\Q(X)\subset\Qeq(X)\subset\NRP^{[1]}(X)$.\end{prop}
\begin{proof}
Let $x_{i},y_{i}\in X$, $g_{i}\in G$ be sequences such that $x_{i}\to x$,
$y_{i}\to y$, $g_{i}x_{i}\to x$ and $g_{i}y_{i}\to x$. As $(G,X)$
is minimal $(x_{i}^{[1]},y_{i}^{[1]})\in C_{G}^{[2]}(X)$. Conclude
$(g_{i}x_{i},x_{i},g_{i}y_{i},y_{i})\in C_{G}^{[2]}(X)$ (using the
identification in Equation (\ref{eq:X2 identification}) in Subsection
\ref{sub:Discrete-cubes}). As $(g_{i}x_{i},x_{i},g_{i}y_{i},y_{i})\to\llcorner^{2}(x,y)$
we have $\llcorner^{2}(x,y)\in C_{G}^{[2]}(X)$ and thus $(x,y)\in\NRP^{[1]}(X)$.
\end{proof}
\begin{defn}
We say $(G,X)$ is a \textbf{homogeneous }t.d.s if and only if $X=K/H$
where $K$ is a compact group, $H\subset K$ is a closed subgroup
and there exists a continuous group homomorphism $\phi:G\to K$ such that
$G$ acts through $gx =g(kH) \phi(g)kH$, where $x =kH \in X =K/H$.
\end{defn}

\begin{thm}
Let $(G,X)$ be a minimal topological dynamical system, then $(G,X/\Qeq(X))$
is the maximal homogeneous factor of $(G,X)$.\end{thm}
\begin{proof}
See \cite[Theorem 1.8]{G03} and \cite[V(1.6)]{dV93}.\end{proof}
\begin{lem}
\label{lem:abelian has trivial RP1}If $(G,K)$ is a minimal abelian
group t.d.s, then $\NRP^{[1]}(K)=\Delta$.\end{lem}
\begin{proof}
Consider $B_{G}^{[2]}(K)\triangleq\{(x,y,z,x+y-z)|\,\ x,y,z\in K\}$.
Notice $B_{G}^{[2]}(K)$ is closed $\Hcal^{[2]}$-invariant and $\{x^{[2]}|\,x\in X\}\subset B_{G}^{[2]}(K)$.
We conclude $C_{G}^{[2]}(K)\subset B_{G}^{[2]}(K)$. Thus $(x,x,x,y)\in C_{G}^{[2]}(K)$
implies $x=y$ and $\NRP^{[1]}(K)=\Delta$.\end{proof}
\begin{thm}
\label{thm:maximal abelian factor}Let $(G,X)$ be a minimal topological
dynamical system, then $\pi_{1}:(G,X)\to(G,X/\NRP^{[1]}(X))$ is the
maximal abelian group factor of $(G,X)$. That is, if $\phi:(G,X)\to(G,K)$
is a factor map where $(G,K)$ is an abelian group t.d.s., then there
exists a map $\psi:(G,X/\NRP^{[1]}(X))\to(G,K)$ such that $\phi=\psi\circ\pi_{1}$. \end{thm}
\begin{proof}
By Theorem \ref{thm:weak structure dynamical}, $(G,X/\NRP^{[1]}(X))$
is an abelian group factor of $(G,X)$. It is enough to show that
$(x,y)\in\NRP^{[1]}(X)$, implies $\phi(x)=\phi(y)$. Indeed by Lemma
\ref{lem:elementary}(5) $(x,y)\in\NRP^{[1]}(X)$ implies $(\phi(x),\phi(y))\in\NRP^{[1]}(K)$
and thus by Lemma \ref{lem:abelian has trivial RP1} $\phi(x)=\phi(y)$.
\end{proof}

\begin{rem}
\label{rem:strong str d=00003D00003D1}It is not hard to show that a compact abelian group is the inverse limit of compact abelian Lie
groups (see \cite[Theorem 5.2(a)]{Sep07}). Thus Theorem \ref{thm:maximal abelian factor}
implies that a minimal t.d.s $(G,X)$ of order $1$ is a pronilsystem
of order $1$. This strengthes Theorem \ref{thm:strong structure}
in the case $d=1.$
\end{rem}
Note if $G$ is not abelian it may happen that $\Qeq(X)\neq\NRP^{[1]}(X)$:
\begin{example}
\label{ex:A5} Let $G=X=A_{5}$, the alternating group on $5$ symbols,
where $G$ acts on $X$ by left multiplication. Clearly the minimal
t.d.s $(G,X)$ is equicontinuous so, $\Q(X)=\Qeq(X)=\triangle$.
As $A_{5}$ is simple, it is perfect. By Lemma \ref{lem:elementary},
$\NRP^{[d]}(X)=X\times X$ for all $d\geq1$.
\end{example}

\subsection{A different generalization of $\RP^{[d]}(X)$}\label{sec:diff gen}

We now present a different higher order generalization of the classical regionally
proximal relation for arbitrary group actions. This definition has the advantage that for
$d=1$ and arbitrary acting group it coincides with the classical definition of $\Q(X)$. Moreover for
$d>1$ and abelian acting group it coincides with $\RP^{[d]}(X)$ as defined by Host, Kra and Maass. Therefore we will keep using the notation $\RP^{[d]}(X)$ for the new definition where we put no restriction on the acting group.
\begin{defn}\label{def:gen RPd}
Let $(G,X)$ be a t.d.s. Let $x,y\in X$. A pair $(x,y)\in X\times X$
is said to be \textbf{regionally proximal of order $d$},
denoted $(x,y)\in\RP^{[d]}(X)$ if there are sequences $f_{i}\in\Fcal^{[d]}$,
$x_{i},y_{i}\in X$, and $a_{*}\in X_{*}^{[d]}$ so that:
\begin{equation}\label{eq:diff RPd}
(f_{i}x_{i}^{[d]},f_{i}y_{i}^{[d]})\rightarrow((x,a_{*}),(y,a_{*})).
\end{equation}
\end{defn}

\begin{prop} Let $(G,X)$ be a minimal t.d.s. Then $\RP^{[d]}(X)\subset \NRP^{[d]}(X).$
\end{prop}
\begin{proof}
Assume that $(x,y)\in\RP^{[d]}(X)$. By definition there are sequences $f_{i}\in\Fcal^{[d]}$,
$x_{i},y_{i}\in X$, and $a_{*}\in X_{*}^{[d]}$ so that (\ref{eq:diff RPd}) holds. Our first goal is to show that
$((x,a_{*}),(y,a_{*}))$ is a cube. Indeed if this is true then $(x,a_{*})\in C_G^{[d]}(X)$, and hence by Corollary \ref{cor:approaching constant} there are $g_{i}\in\Fcal^{[d]}$ such that $g_i(x,a_{*} )\rightarrow x^{[d]}$. Thus by doubling (see Subsection \ref{sub:doubling}), it follows that $(g_i(x,a_{*}),g_i(y,a_{*}))\rightarrow (x^{[d]},y,x^{[d]}_*)\in C_G^{[d+1]}(X)$ which implies that $(x,y)\in \NRP^{[d]}(X)$ as desired.

To show that $((x,a_{*}),(y,a_{*}))\in C_G^{[d+1]}(X)$, we note that as $(G,X)$ is minimal, we have $(x_i^{[d]},y_i^{[d]})\in C_G^{[d+1]}(X)$. Thus again by doubling $(f_{i}x_{i}^{[d]},f_{i}y_{i}^{[d]})\in C_G^{[d+1]}(X)$ and it follows.
\end{proof}

\begin{rem} Let us look at Example \ref{ex:A5} again. We know that $\NRP^{[d]}(X)=X\times X$ for all $d\geq1$.
At the same time $\Q(X)=\triangle$. This implies that $\RP^{[d]}(X)=\Delta$, as $\RP^{[d]}(X)\subset \RP^{[1]}(X)=\Q(X)$
by Lemma \ref{lem:elementary}. Thus, for this system, $\NRP^{[d]}(X)\not=\RP^{[d]}(X)$ for all $d\ge 1$.

%Let $G=X=A_{5}$, the alternating group on $5$ symbols,
%where $G$ acts on $X$ by left multiplication. Clearly the minimal
%t.d.s $(G,X)$ is equicontinuous so, $\Q(X)=\Qeq(X)=\triangle$.
%As $A_{5}$ is simple, it is perfect. By Lemma \ref{lem:elementary},
%$\NRP^{[d]}(X)=X\times X$ for all $d\geq1$.

%If $G$ is a finite non-abelian group, then it is easy to see that for the
%minimal system $(G,G)$, $\RP^{[1]}(X)\not=\NRP^{[1]}(X)$, since $X/\RP^{[1]}(X)$ is the maximal
%equicontinuous factor and at the same time $X/\NRP^{[1]}(X)$ is the maximal abelian factor
%by Theorem \ref{thm:maximal abelian factor}.
\end{rem}

\section{A minimal system which does not induce a fibrant cubespace\label{sec:A-dynamical-cubespace}}

According to Theorem \ref{thm:completion for distal} a minimal distal
action induces a fibrant cubespace. Here we exhibit an example of
a non-distal minimal $\mathbb{{Z}}$-system which is not fibrant.
This is proven by showing that a weaker property, the so-called glueing
property, fails to hold for this system. We start by a definition
and a proposition:
\begin{defn}
We say a cubespace $(X,{C^{\bullet}})$ has the \textbf{glueing property}
if ``glueing'' two cubes along a common face yields another cube.
Formally, let $d\geq1$ and suppose $c,c'\in C_{G}^{[d]}(X)$, $c=(c_{1},c_{2})$
and $c'=(c_{2},c_{3})$, then $(c_{1},c_{3})\in C_{}^{d}(X)$.\end{defn}
\begin{prop}
\label{prop:fibrant implies gluing}If a cubespace \textup{$(X,{C^{\bullet}})$}
is fibrant then it has the glueing property.\end{prop}
\begin{proof}
See \cite[Proposition 6.2]{GMVI}.\end{proof}
\begin{example}
We now present an example of a non-distal minimal $\mathbb{{Z}}$-system
which is not fibrant. This example is closely related to the examples
given in \cite[p. 254]{G94} and \cite[Example 3.6]{tu2013dynamical}.
Let $\mathbb{S}^{1} \cong \mathbb{R}/\mathbb{Z}$
be the circle group, also identified with the interval $[0,1]$
with identified endpoints. Let $T:\mathbb{S}^{1}\rightarrow\mathbb{S}^{1}$
be the rotation by an irrational  number $\alpha$,
$Tx=x+\alpha \pmod 1$. This is a minimal and equicontinuous system.
Let $H_{1}=[0,\frac{1}{2}]$
and $H_{0}=[\frac{1}{2},1]$ be subsets of $\mathbb{S}^{1}$. Define
$f(n)=\chi_{H_{0}}(\{n\alpha\})$ for $n\in\Z$ where $\{n\alpha\}=n\alpha\pmod1$.
We consider $f$ as an element in the full shift on two letters and
define $X$ to be its orbit closure, i.e.:
\[
X=\overline{\Or(f,\mathbb{{Z}})}\subset\{0,1\}^{\mathbb{{Z}}}
\]
We will denote the shift on $\{0,1\}^{\Z}$ by $\sigma$. The system
$(X,\sigma)$ is a particular example of a \textit{Sturmian-like system}
(for an introduction to these systems see \cite[p.239]{A}). Define
the following natural dynamical morphism $\pi:(X,\sigma)\rightarrow(\mathbb{S}^{1},T)$
by $\pi((x)_{n\in\mathbb{{Z}}})=\bigcap_{n\in\mathbb{{Z}}}T^{-n}H_{x_{n}}$.
Note that for all $x\in X$, the intersection consists of one element
exactly of the circle so the map is well defined and continuous. Moreover
for any element of the circle which does not belong to the orbit of
$0$ or $\frac{1}{2}$, i.e. for $x\notin E=\bigcup_{n\in Z}T^{n}\{0,\frac{1}{2}\}$, we have $|\pi^{-1}(x)|=1$. For $x\in E$ one has $|\pi^{-1}(x)|=2$. This immediately implies that $(X,\sigma)$ is minimal. Denote by $0^{+},0^{-}$ the preimages of $0$ under $\pi$ , then $0^{+},0^{-}$ are proximal
as they differ only at the zeroth coordinate. To be specific let us decide that $0^{+}(0)=1$ and $0^{-}(0)=0$. Let us equip the circle $\mathbb{S}^{1}$ with the anti-clockwise orientation. Given
two pairs of points $(x_{1},y_{1})$, $(x_{2},y_{2})$ with $|x_{i}-y_{i}|<\frac{1}{2}$,
we may thus compare their orientations. Define $U_{+}=\bigcap_{n=0}^{n=1}T^{-n}H_{0^{+}(n)}$
and $U_{-}=\bigcap_{n=0}^{n=1}T^{-n}H_{0^{-}(n)}$. Clearly $0^{+}\in U_{+}$
and $0^{-}\in U_{-}$ and $U_{+}\cap U_{-}=\emptyset$. Moreover there
exists some $\epsilon>0$ such that $[0,\epsilon)\subset\pi(U_{+})$
and $(1-\epsilon,1]\subset\pi(U_{-})$. Let $z\in X$. By minimality
$(0^{+},0^{-},0^{+},0^{-}),(0^{-},0^{+},0^{-},0^{+})\in C_{\Z}^{[2]}(X)$, where we use convention (\ref{eq:X2 identification}). Thus by proximality
of the pair $(0^{+},0^{-})$ it follows that $(0^{+},0^{-},z,z),(0^{-},0^{+},z,z)\in C_{\Z}^{[2]}(X)$.
Assume for a contradiction that $(X,{C_{\Z}^{\bullet}})$ is fibrant.
By Proposition \ref{prop:fibrant implies gluing}, gluing $(0^{+},0^{-},z,z)$ and $(0^{-},0^{+},z,z)$, we have $(0^{+},0^{-},0^{-},0^{+})\in C_{\Z}^{[2]}(X)$.
By definition of $C_{\Z}^{[2]}(X)$ , one may find sequences $w_{i},y_{i}\in X$
and $n_{i}\in\Z$ such that
\[
(w_{i},y_{i},\sigma^{n_{i}}w_{i},\sigma^{n_{i}}y_{i})\rightarrow_{i\rightarrow\infty}(0^{+},0^{-},0^{-},0^{+})
\]
Note that for big enough $i$, $(\pi(\sigma^{n_{i}}w_{i}),\pi(\sigma^{n_{i}}y_{i}))$
is oriented as $(\pi(w_{i}),\pi(y_{i}))$.
However $\pi(w_{i})\in[0,\epsilon)$
and $\pi(y_{i})\in(1-\epsilon,1]$, whereas $\pi(\sigma^{n_{i}}w_{i})\in(1-\epsilon,1]$
and $\pi(\sigma^{n_{i}}y_{i})\in[0,\epsilon)$. Contradiction.

\begin{figure}[h]
    \centering
    \includegraphics[width=1.5in]{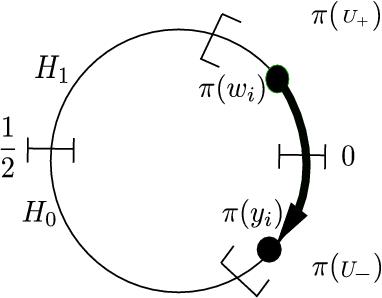}
    \caption{}
     \label{fig:circle}
\end{figure}
\end{example}

%\newpage

\section{Open questions}\label{sec:Open-Questions}
\subsection{Questions relating to $\NRP^{[d]}$}

In Theorem \ref{thm:strong structure}
one assumes that $G$ has a dense subgroup generated by a compact set.
We thus ask: \begin{question} For which groups $G$ does Theorem
\ref{thm:strong structure} hold? \end{question}

Note that given Theorem \ref{thm:maximal d factor}, this is equivalent
to the following question:
\begin{question} Let $d\geq2.$ For which
groups $G$ is the maximal factor of $(G,X)$ of order at most $d$
a pronilsystem? \end{question}
Note that for $d=1$, Remark \ref{rem:strong str d=00003D00003D1}
gives a complete solution to this question.

As an intermediate step one can try to answer the following question:

\begin{question}
Let $(G,X)$ be a minimal system of finite order. Is it uniquely ergodic?
\end{question}

By Theorem \ref{thm:maximal abelian factor}, for any minimal topological
dynamical system, $(G,X/\NRP^{[1]}(X))$ is the maximal abelian group
factor of $(G,X)$. Thus the following problem is natural:
\begin{problem}
Find an explicit description, for minimal topological
dynamical systems $(G,X)$, of the equivalence relation $\RR(X)$ such that $(G,X/\RR(X))$ is the maximal
(compact) group factor of $(G,X)$.
\end{problem}
\subsection{Questions relating to $\RP^{[d]}$}
The following questions refer to $\RP^{[d]}$ as defined in Section \ref{sec:diff gen}.
\begin{question}
Let $(G,X)$ be a minimal t.d.s where $G$ is not
abelian. Assume the Bronstein condition (see Subsection \ref{sub:Proximality-and-its})
holds or that $G$ is amenable. Is $\RP^{[d]}(X)$ an equivalence
relation for $d\geq2$? %Is it true that $\RP^{[d]}(X)=\NRP^{[d]}(X).$
\end{question}
Note that for $d=1$ the answer is known to be positive for the first question.

\begin{question}
Let $(G,X)$ be a minimal t.d.s, $d\in \mathbb{N}$ and $\RP^{[d]}(X)$ is an equivalence
relation. What can be said about the structure of $X/\RP^{[d]}(X)$ when $G$ is not abelain?
\end{question}

Note that by Lemma \ref{lem:elementary}, $X/\RP^{[d]}(X)$ is a distal system when
$\RP^{[d]}(X)$ is an equivalence relation.

\appendix\label{sec:Appendix}

\renewcommand{\thesection}{\hspace*{-8pt}}
%\section{}
\section{}
\renewcommand{\thesection}{\Alph{section}}
\renewcommand{\thelem}{\Alph{lemma}}

\subsection{Cube invariance\label{sub:cube invariance}}

We verify a claim made in Subsection \ref{sub:Nilspaces}:
\begin{prop}\label{prop ap:cube inv}
\label{prop:dynamical cubespaces are cubespaces} Let $(G,X)$ be
a topological dynamical system and let $(X,{C_{G}^{\bullet}})$ be
the dynamical cubespace induced by (G,X) . Then $(X,{C_{G}^{\bullet}})$
has cube invariance.\end{prop}
\begin{proof}
From the definition of $C_{G}^{[d]}(X)$ in Equation (\ref{eq:def of Cn H})
in Subsection \ref{sub:Dynamical-cubespaces}, it is clearly
% it is
enough
to prove that for any $g\in\Hcal^{[d]}$ and morphism of discrete
cubes $f:\{0,1\}^{r}\to\{0,1\}^{d}$ we have $g\circ f\in\Hcal^{[r]}$.
We can assume without loss of generality that $g=[h]_{F}$ for $h\in G$
and $F=\{\o\in\{0,1\}^{d}|\ \o=a\}$ a hyperface of $\{0,1\}^{d}$,
where $a\in\{0,1\}$ and $t\in\{1,2,\ldots,d\}$. Let us write explicitly
$f=(f_{1},\ldots,f_{d})$ where $f_{j}(\o_{1},\ldots,\o_{r})$ equals
to either $0$, $1$, $\o_{i}$ or $\overline{o_{i}}=1-\o_{i}$ for some $1\le i=i(j)\le r$.
Denote $H=f^{-1}(F)$, then $H$ is the face of of $\{0,1\}^{r}$.
If $f_{t}\equiv a$, then $H=\{0,1\}^{r}$, if $f_{t}\equiv1-a,$
then $H=\emptyset$, otherwise $H$ is a hyperface of $\{0,1\}^{r}$.
We conclude $g\circ f=[h]_{H}\in\Hcal^{[r]}$.
\end{proof}

\subsection{Doubling\label{sub:doubling}}

Consider the morphisms of discrete cubes $\hat{\pi}_{i}:\{0,1\}^{d+1}\to\{0,1\}^{d}$,
$i=1,\ldots,d+1$ defined by

\[
\hat{\pi}_{i}(\epsilon_{1},\epsilon_{2},\ldots,\epsilon_{i},\ldots,\epsilon_{d+1})=(\epsilon_{1},\epsilon_{2},\ldots,\hat{\epsilon_{i}},\ldots,\epsilon_{d+1})
\]
Let $F^{s}_{i}=\{\o\in\{0,1\}^{s}|\ \o_{i}=1\}$. Note $\hat{\pi}_{i}^{-1}(F^{d}_{j})=F^{d+1}_{j}$ if $j<i$, and $\hat{\pi}_{i}^{-1}(F^{d}_{j})=F^{d}_{j+1}$ if $j>i$.
Define $D_{i}(f)(\omega)=f(\hat{\pi}_i(\omega))$ for $f\in \Fcal^{[d]}$ and $\omega\in \{0,1\}^{d+1}$.

\begin{lem}\label{lem:doubling}
$D_{i}(\Fcal^{[d]})\subset \Fcal^{[d+1]}$.
\end{lem}
\begin{proof}
By the definition of $\Fcal^{[d]}$ in Subsection \ref{sub:Host-Kra-cube-group}, it is enough to note for $h\in G$, $D_{i}([h]_{F^{d}_{j}})=[h]_{\hat{\pi}_{i}^{-1}(F^{d}_{j})}\in \Fcal^{[d+1]}$ for $j=1,\ldots,d$.
\end{proof}
In fact we see that  $D_{i}(f)$ consists of ``painting'' $f$ on $F_{i}^{d+1}$ and on
the corresponding parallel lower hyperspace $(F_{i}^{d+1})^{c}$. We refer
to this operation as \textbf{doubling along $F_{i}$}. Notice that
using our convention in Equation (\ref{eq:product}) of Subsection
\ref{sub:Discrete-cubes} we have $D_{d+1}(f)=f\times f$.

\begin{figure}[h]
\centering{}\begin{tikzpicture}
\begin{scope}[shift={(-2,0.7)}]         \singlesquare{$x_{00}$}{$x_{01}$}{$x_{10}$}{$x_{11}$}
\end{scope}
   \draw[->] (0,1.2) -- (1,1.2);
\begin{scope}[shift={(2,0)}]         \threecube{$x_{00}$}{$x_{01}$}{$x_{10}$}{$x_{11}$}{$x_{00}$}{$x_{01}$}{$x_{10}$}{$x_{11}$}{1}
 \begin{scope}[x={(1, 0)}, y={(0, 1)}, z={(0.352, 0.317)}, scale=2]           \draw[ultra thick] (0,0,0) -- (0,1,0);
         \draw[ultra thick] (1,0,0) -- (1,1,0);
        \draw[ultra thick] (0,0,1) -- (0,1,1);
       \draw[ultra thick] (1,0,1) -- (1,1,1);
     \end{scope}
 \end{scope}
\end{tikzpicture} \caption{Doubling along $F_{3}$.}
\end{figure}

\subsection{Pure ceiling and mixed upper faces}

Let $d\geq1$ and let $F$ be an \textit{upper} face (see Subsection
\ref{sub:Faces}). If $F$ is contained in the \textit{ceiling hyperface}
$F=\{\o\in\{0,1\}^{d}|\ \o_{_{d}}=1\}$ we call it \textbf{pure ceiling}.
Otherwise we call it \textbf{mixed}. Note there are $2^{d-1}$ pure
ceiling faces and $2^{d-1}$ mixed faces. Fix $g\in G$, pure ceiling
face $P$ and mixed face $M$. Note:

\begin{equation}
[g]_{P}=\Id^{[d-1]}\times[g]_{L_{1}}\label{eq:id times}
\end{equation}

\begin{equation}
[g]_{M}=[g]_{L_{2}}\times[g]_{L_{2}}\label{eq:g times g}
\end{equation}
where $L_{1},L_{2}$ are some upper faces of $\{0,1\}^{d-1}$.

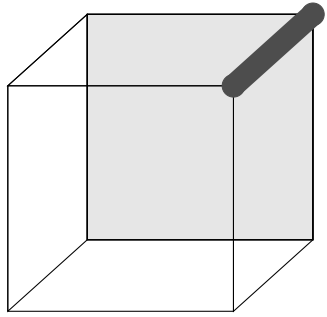
\begin{figure}[h]
\centering{}\begin{tikzpicture}
\begin{scope}
[x={(1, 0)}, y={(0, 1)}, z={(0.352, 0.317)}, scale=3]
 \draw[fill=black!10] (1,0,1) -- (0,0,1) -- (0,1,1) -- (1,1,1) -- cycle;
 \draw[thin] (1,0,1) -- (0,0,1) -- (0,1,1) -- (1,1,1) -- cycle;
 \draw[thin] (1,1,1) -- (0,1,1) -- (0,0,1) -- (1,0,1) -- cycle;
  \draw[thin] (0,0,0) -- (1,0,0) -- (1,1,0) -- (0,1,0) -- cycle;
   \draw[thin] (1,1,1) -- (1,1,0) -- (1,0,0) -- (1,0,1) -- cycle;
   \draw[thin] (1,1,1) -- (0,1,1) -- (0,1,0) -- (1,1,0) -- cycle;
   \draw[thin] (0,0,0) -- (0,0,1) -- (0,1,1) -- (0,1,0) -- cycle;
   \draw[line width=0.3cm,black!70] (1,1,0) -- (1,1,1);
  \draw[fill=black!70,black!70] (1,1,1) circle[radius=0.05cm];
  \draw[fill=black!70,black!70] (1,1,0) circle[radius=0.05cm];
\end{scope}
  \end{tikzpicture} \caption{An example of a pure ceiling face and a mixed face.}
\end{figure}

\subsection{Lower central series induced representation for the Host-Kra cube
group \label{sub:Lower-central-series}}

\noindent Let $G$ be a group. Set $G=G_{0}=G_{1}$ and define inductively
$G_{i+1}=[G,G_{i}]$, where for $A,B\subset G$, $[A,B]$ is the group
generated by the commutators $[a,b]$, $a\in A$, $b\in B$. The sequence

\[
G=G_{0}=G_{1}\supseteq G_{2}\supseteq\ldots
\]
is called the \textbf{lower central series} of $G$.

If $F$ is a face of codimension $d$ and $d_{1},d_{2}$ are positive
integers with $d_{1}+d_{2}=d$ then we can find faces $F_{1}$ and
$F_{2}$ of codimension $d_{1}$ and $d_{2}$, respectively, such
that $F_{1}\cap F_{2}=F$. Note the following key equality:
\begin{equation}
[[g_{1}]_{F_{1}},[g_{2}]_{F_{2}}]=[[g_{1},g_{2}]]_{F}\label{eq:key commutator}
\end{equation}
We conclude that the Host-Kra cube group $\Hcal^{[d]}$ is generated
by $[G_{\co(F)}]_{F}$ where $F$ ranges over all faces of $\{0,1\}^{d}$.

\subsection{The pure ceiling-mixed decomposition\label{sub:HK decomposition}}
\begin{lem}
\noindent \label{lem:F representation}%
Let $d\geq1$, and fix an ordering $<$ on $S_{1}=\{\vec{1}\},S_{2},\dots,S_{2^{d}}=\{0,1\}^{d}$
of the upper faces that respects inclusion, i.e.~if $S_{i}\subsetneq S_{j}$
then $S_{i}<S_{j}$. Then any element $g\in\Fcal^{[d]}$ has a representation
as an ordered product $[x_{1}]_{S_{1}}[x_{2}]_{S_{2}}\cdots[x_{2^{d}}]_{S_{2^{d}-1}}$where
for $1\le i\le2^{d}-1$, $x_{i}\in G$ and $[x_{i}]_{S_{i}}\in\Fcal^{[d]}$. \end{lem}
\begin{proof}
This is essentially proven in \cite[Proposition A.5]{GMVI} (see also
\cite[Appendix E]{GT10})\footnote{One can actually prove that $\Hcal^{[d]}=[G_{\co(S_{1})}]_{S_{1}}[G_{\co(S_{2})}]_{S_{2}}\cdots[G_{\co(S_{2^{d}})}]_{S_{2^{d}}}$, where $G=G_{0}=G_{1}\supseteq G_{2}\supseteq\ldots$ is the lower
central series of $G$ (see Subsection \ref{sub:Lower-central-series}).
In addition the induced representation for elements in $\Hcal^{[d]}$
is unique but we will not need these facts.}. Let us sketch the proof. Fix $S_{i}<S_{j}$ and $g,h\in G$ with.
By (\ref{eq:key commutator}) we have $[g]_{S_{j}}[h]_{S_{i}}=[g,h]_{S_{i}\cap S_{j}}[h]_{S_{i}}[g]_{S_{j}}$
where $[g,h]=ghg^{-1}h^{-1},$ as $[g,h]_{S_{i}\cap S_{j}}=[[g]_{S_{j}},[h]_{S_{i}}]$.
In other words

\begin{equation}
[G]_{S_{j}}[G]_{S_{i}}\subset[G]_{S_{k}}[G]_{S_{i}}[G]_{S_{j}}\label{eq:commutator}
\end{equation}
for some $S_{k}$ for which $S_{k}\leq S_{i}$ and $S_{k}<S_{j}$.
By definition any $g\in\Fcal^{[d]}$ is of the form $\prod_{j=1}^{m}[t_{j}]_{F_{j}}$
where $F_{j}$ is an upper hyperface and $t_{j}\in G$. Thus one can
use (\ref{eq:commutator}) to move all occurrences of elements of
the form $[G]_{S_{2^{d}-1}}$ to the far right, then move all occurrences
of elements of the form $[G]_{S_{2^{d}-2}}$ to be adjacent to $[G]_{S_{2^{d}-1}}$, and so on so as to establish $g=[x_{1}]_{S_{1}}[x_{2}]_{S_{2}}\cdots[x_{2^{d}}]_{S_{2^{d}-1}}$.

\end{proof}
\begin{prop}
\label{prop:H=00003Dpure ceiling mixed}Let $G$ be a group and $d\geq1$.
If $g\in\Hcal^{[d]}$ then there are elements $(s\times s)\in\Hcal^{[d]}$
and $(\id^{[d-1]}\times h)\in\Fcal^{[d]}$ such that $g=(\id^{[d-1]}\times h)(s\times s)$
for some $h,s\in G^{[d-1]}$.
\end{prop}
By Proposition \ref{prop:H=00003DFDiag} there are $f\in\Fcal^{[d]}$
and $h\in G$ so that $g=f[h]_{\{0,1\}^{d}}$. Fix an ordering $<$
on $S_{1}=\{\vec{1}\},S_{2},\dots,S_{2^{d}}=\{0,1\}^{d}$ of the upper
faces that respects inclusion, i.e.~if $S_{i}\subsetneq S_{j}$ then
$S_{i}<S_{j}$. Moreover assume that if $P$ is a pure ceiling upper
face and $M$ is a mixed upper face then $P<M$ (this is possible
as a pure ceiling upper face cannot contain a mixed upper face). By
Lemma \ref{lem:F representation} we may write:
\begin{equation}
\underbrace{f=[x_{1}]_{S_{1}}[x_{2}]_{S_{2}}\cdots[x_{2^{[d-1]}}]_{S_{2^{d-1}}}}_{\mathrm{{pure}\ ceiling}}\underbrace{[x_{2^{d-1}+1}]_{S_{2^{d-1}+1}}[x_{2^{d-1}+2}]_{S_{2^{d-1}+2}}\cdots[x_{2^{d}}]_{S_{2^{d}-1}}}_{\mathrm{{mixed}}}\label{eq:pure mixed decomposition}
\end{equation}
Note that by Equation (\ref{eq:id times}) the product $[x_{1}]_{S_{1}}[x_{2}]_{S_{2}}\cdots[x_{2^{[d-1]}}]_{S_{2^{d-1}}}\in\Fcal^{[d]}$
is of the form $(\id^{[d-1]}\times h)$ where $h\in G^{[d-1]}$, whereas
by Equation (\ref{eq:g times g}) the product $[x_{2^{d-1}+1}]_{S_{2^{d-1}+1}}[x_{2^{d-1}+2}]_{S_{2^{d-1}+2}}\cdots[x_{2^{d}}]_{S_{2^{d}}}\in\Fcal^{[d]}$
is of the form $(s'\times s')$ where $s'\in G^{[d-1]}$.

\subsection{Elementary properties of $\NRP^{[d]}(X)$.}
\begin{lem}
\label{lem:elementary}Let $(G,X)$ be a minimal t.d.s then:\end{lem}
\begin{enumerate}
\item $\PP(X)\subseteq\cdots\subseteq\NRP^{[d+1]}(X)\subseteq\NRP^{[d]}(X)$ for each $d\in\N$.
\item
$\PP(X)\subseteq\cdots\subseteq\RP^{[d+1]}(X)\subseteq\RP^{[d]}(X)$ for each $d\in\N$.
\item If $\NRP^{[d]}(X)=\Del$ for some $d\geq1$ then $(G,X)$ is distal.
\item If $G_{d+1}$ denotes the $(d+1)$-th element of the lower central
series of $G$, then $(x,hx)\in\NRP^{[d]}(X)$ for any $h\in\overline{G_{d+1}}$.
Hence if $G$ is perfect, that is $G=[G,G]$, then $\NRP^{[d]}(X)=X\times X$
for all $d\geq1$.
\item If $\pi:(G,X)\to(G,Y)$ is a dynamical morphism then $\pi\times\pi(\NRP^{[d]}(X))\subseteq\NRP^{[d]}(Y)$.\end{enumerate}
\begin{proof}
\begin{enumerate}
\item As $\pi_{f}(C_{G}^{[d+2]}(X))=C_{G}^{[d+1]}(X)$ it follows directly
from Definition \ref{def:general RPd} that $\NRP^{[d+1]}(X)\subseteq\NRP^{[d]}(X)$
for $d\geq1$. By Proposition \ref{prop:Q subset RP1}, $\PP(X)\subset\NRP^{[1]}(X)$. We now
proceed by induction to show that $\PP(X)\subseteq \NRP^{[d]}(X)$ for each $d\in\N$. Let $(x,y)\in\PP(X)$. Assume $(x,y)\in\NRP_{}^{[d]}(X)$
which implies $\llcorner^{[d+1]}(x,y)\in C_{G}^{[d+1]}(X)$. By cube
invariance (see Subsection \ref{sub:Nilspaces}), $c\triangleq(\llcorner^{[d+1]}(x,y),\llcorner^{[d+1]}(x,y))\in C_{G}^{[d+2]}(X)$.
Let $F=\{\o\in\{0,1\}^{d+2}:\o_{d+2}=0\}$. Note $c_{|F}=\llcorner^{[d+1]}(x,y)$.
As $(G,X)$ is minimal and $(x,y)\in\PP(X)$ one may find a sequence
$g_{i}\in G$ such that $g_{i}x\to x$ and $g_{i}y\to x$. Conclude
$([g_{i}]_{F})c=(g_{i}^{[d+1]}\llcorner^{[d+1]}(x,y),\llcorner^{[d+1]}(x,y))\to\llcorner^{d+2}(x,y)\in C_{G}^{[d+2]}(X)$
which implies $(x,y)\in\NRP_{}^{[d+1]}(X)$ as desired.
\item
Consider the floor map $\pi_{f}:G^{[d+1]}\to G^{[d]}$ from Subsection \ref{sub:Discrete-cubes}. Let us denote its restriction to $\Fcal^{[d+1]}$ by $\phi_{f}$. Clearly, $\phi_{f}(\Fcal^{[d+1]})=\Fcal^{[d]}$. It follows from Definition \ref{def:gen RPd} that $\RP^{[d+1]}(X)\subseteq\RP^{[d]}(X)$
for each $d\in\N$. Now we show that $\PP(X)\subseteq \RP^{[d]}(X)$ for each $d\in\N$. It follows by the definition that
$\PP(X)\subset\RP{}^{[1]}(X)$ as $(id,g,id,g)\in \Fcal^{[2]}$ for each $g\in G$.
% and there are $g_i\in G$ such that $g_ix\rightarrow x$ and $g_iy\rightarrow x$.
We now proceed by induction. Let $(x,y)\in\PP(X)$ and assume $(x,y)\in \RP{}^{[d]}(X)$ which implies that
there are sequences $f_{i}\in\Fcal^{[d]}$, $x_i,y_i\in X$ with and $a_{*}\in X_{*}^{[d]}$ with
$(f_{i}x_i^{[d]},f_{i}y_i^{[d]})\rightarrow((x,a_{*}),(y,a_{*}))$. As part of the induction one may assume $x_i=x,y_i=y$ for all $i$. Since $(x,y)\in \PP(X)$ and $(X,G)$ is minimal,
there are $g_i\in G$ such that $g_ix\rightarrow x$ and $g_iy\rightarrow x$. There is a subsequence $\{n_i\}$
such that $g^{[d]}_{n_i}f_{n_i}x^{[d]}\rightarrow (x,b_*)$ and $g^{[d]}_{n_i}f_{n_i}y^{[d]}\rightarrow (x,b_*)$,
here $b_*=\lim g_{n_i}a_*$. Thus
$$(id^{[d]},g_{n_i}^{[d]})\cdot (f_{n_i},f_{n_i})(x^{[d+1]})=(f_{n_i}x^{[d]}, g_{n_i}^{[d]}f_{n_i}x^{[d]})\rightarrow (x,a_*,x,b_*)$$
and
$$ (id^{[d]},g_{n_i}^{[d]})\cdot (f_{n_i},f_{n_i})(y^{[d+1]})=(f_{n_i}y^{[d]}, g_{n_i}^{[d]}f_{n_i}y^{[d]})\rightarrow (y,a_*,x,b_*).$$
It is clear that $(id^{[d]},g_{n_i}^{[d]})\cdot (f_{n_i},f_{n_i})\in \Fcal^{[d+1]}$, and the result follows.

\item By (1) $\NRP^{[d]}(X)=\Del$ implies $\PP(X)=\Del$.

\item Follows as $\Hcal^{[d]}$ is generated by $[G_{\co(F)}]_{F}$ where
$F$ ranges over all faces of $\{0,1\}^{d}$ (see Subsection \ref{sub:Lower-central-series}).

\item Follows directly from Definition \ref{def:general RPd}.
\end{enumerate}
\end{proof}

For the next proposition recall the discussion in Example \ref{ex:weakly mixing}.

\begin{prop}\label{prop:weakly mixing}
Let $(G,X)$ be a minimal t.d.s which is transitive of all orders, then:
\begin{enumerate}
\item For all $x\in X$ and $d\in \mathbb{N}$, $Y_{x}^{[d]}(X)={x}\times X_{*}^{[d]}$.
\item For all $x\in X$ and $d\in \mathbb{N}$, $Y_{x}^{[d]}(X)=C_{x}^{[d]}(X)$.
\item For all $x\in X$ and $d\in \mathbb{N}$, $\NRP^{[d]}(X)=X\times X$.
\end{enumerate}
\end{prop}

\begin{proof}
We start by proving $(1)$ by induction. Fix $x\in X$. The case $d=1$ follows from minimality.
Assume the statement for $d-1$, $d\ge 2$. Note this implies $(2)$ for $d-1$ and thus $C_{G}^{[d-1]}(X)=X^{[d-1]}$. Let $ a\in X^{[d-1]}$ be a transitive point. By Proposition \ref{prop:minimality}, $C_{G}^{[d-1]}(X)=X^{[d-1]}$ is $\Hcal^{[d-1]}$-minimal. We may thus find a sequence $g_k\in \Hcal^{[d-1]}$ such that $g_k x^{[d-1]}\rightarrow a$. By Proposition \ref{prop:H=00003DFDiag}, there is a sequence $f_k\in \Fcal^{[d-1]}$ and $h\in G$ so that $g_k=f_k h^{[d-1]}$. Note $(f_k\times f_k)(\Id^{[d-1]}\times h^{[d-1]})\in \Fcal^{[d]}$. By passing to a subsequence there is $w\in Y_{x}^{[d-1]}(X)$ so that $(f_k\times f_k)(\Id^{[d-1]}\times h^{[d-1]})(x^{[d-1]},x^{[d-1]})\rightarrow (w,a)$ and we conclude $(w,a)\in Y_{x}^{[d]}(X)$. Note that for any $h\in G$, $(\Id^{[d-1]}\times h^{[d-1]})(w,a)=(w,(h^{[d-1]})a)\in Y_{x}^{[d]}(X)$. Since the element $a$ is a transitive point, we have
\begin{equation}\label{d1}
\{w\}\times X^{[d-1]}\subset Y_{x}^{[d]}(X).
\end{equation}
By Proposition \ref{prop:Yx is minimal}, $w$ is $\Fcal^{[d-1]}$-minimal and
\begin{equation}\label{d2}
    \overline{\Fcal^{[d-1]}}(w)=Y_{x}^{[d-1]}(X)= \{x\}\times X^{[d-1]}_*.
\end{equation}
By acting the elements of $\Fcal^{[d]}$ on (\ref{d1}) and doubling (see Subsection \ref{sub:doubling}), we have
\begin{equation}\label{d3}
    \overline{\Or(w, \Fcal^{[d-1]})} \times X^{[d-1]}\subset Y_{x}^{[d]}(X).
\end{equation}
By (\ref{d2}) and (\ref{d3}), we have
$$\{x\}\times X^{[d-1]}_*\times X^{[d-1]}=\{x\}\times X^{[d]}_*\subset Y_{x}^{[d]}(X).$$
This completes the proof of $(1)$ for $d$. Finally trivially $(1)\Rightarrow (2)\Rightarrow (3)$.
\end{proof}

Let us call two t.d.s $(G,X)$ and $(G',X')$, where possibly $G\neq G'$, \textbf{isomorphic} if there exist a continuous surjective (but not necessarily injective) group homomorphism $\phi:G\to G'$  and a homeomorphism $f:X\to X'$ such that for all $x\in X$ and $g\in G$, $f(gx)=\phi(g)f(x)$. Let $\Fix(G,X)=\{g\in G|\,\forall x\in X,\, gx=x\}$. It is easy to see $\Fix(G,X)$ is a closed subgroup of $G$ and $(G,X)$ and $(G/\Fix(G,X),X)$ are isomorphic.

\begin{prop}\label{prop:effective action}
Let $(G,X)$ be a system of order at most $d$, i.e., $\NRP^{[d]}(X)=\Del$, and denote by $G_{d+1}$  the $(d+1)$-th element of the lower central series of $G$, then $(G,X)$ is isomorphic to $(H,X)$, where $H=G/\overline{G_{d+1}}$ is a nilpotent topological group of nilpotency class at most $d$.
\end{prop}

\begin{proof}
By Lemma \ref{lem:elementary}(4) for all $x\in X$ and $g\in\overline{G_{d+1}}$, $(x,gx)\in\NRP^{[d]}(X)$ which by assumption implies $gx=x$. By \cite[Lemma 5.1]{MKS66} the elements of the lower central series of $G$ are normal in $G$. Thus  $\overline{G_{d+1}}$ is normal in $G$ and $H=G/\overline{G_{d+1}}$ is a topological group. We conclude $(G,X)$ is isomorphic to $(H,X)$. Given a group homomorphism $G'\to H'$ the lower central series of $G'$ is mapped onto the lower central series of $H'$. Thus for $H=G/\overline{G_{d+1}}$, $H_{d+1}=\{\Id\}$ and $H$ is a nilpotent group of nilpotency class at most $d$.
\end{proof}

\begin{prop}\label{prop:alt def of nrp}
$(x,y)\in\NRP^{[d]}(X)$ if and only if $\urcorner^{[d+1]}(x,y)\in C_{G}^{[d+1]}(X)$.
\end{prop}

\begin{proof}
By Theorem \ref{thm:main}, $(x,y)\in\NRP^{[d]}(X)$ iff $(y,x)\in\NRP^{[d]}(X)$ iff $\llcorner^{[d+1]}(y,x)\in C_{G}^{[d+1]}(X)$. By cube-invariance (e.g applying $\o_{1},\ldots,\o_{r}\leftrightarrow \overline{\o_{1}},\ldots,\overline{\o_{r}}$) $\llcorner^{[d+1]}(y,x)\in C_{G}^{[d+1]}(X)$ iff $\urcorner^{[d+1]}(x,y)\in C_{G}^{[d+1]}(X)$.
\end{proof}

\bibliographystyle{alpha}
\bibliography{universal_bib}

\newcommand{\etalchar}[1]{$^{#1}$}
\def\cprime{$'$} \def\cprime{$'$}
\begin{thebibliography}{AHG{\etalchar{+}}63}

\bibitem[AAG08]{akin2008topological}
Ethan Akin, Joseph Auslander, and Eli Glasner.
\newblock {\em The topological dynamics of Ellis actions}.
\newblock American Mathematical Soc., 2008.

\bibitem[ACS12]{CS12}
Omar Antol{\'\i}n~Camarena and Balazs Szegedy.
\newblock Nilspaces, nilmanifolds and their morphisms.
\newblock Preprint. http://arxiv.org/abs/1009.3825, 2012.

\bibitem[AHG{\etalchar{+}}63]{auslander1963flows}
Louis Auslander, F~Hahn, L~Green, Lawrence Markus, and W~Massey.
\newblock {\em Flows on homogeneous spaces}.
\newblock Number~53. Princeton University Press, 1963.

\bibitem[Aki]{AkinPrivate}
Ethan Akin.
\newblock Private communication.

\bibitem[Aki10]{akin2010general}
Ethan Akin.
\newblock {\em The general topology of dynamical systems}, volume~1.
\newblock American Mathematical Soc., 2010.

\bibitem[Aus88]{A}
Joseph Auslander.
\newblock {\em Minimal flows and their extensions}, volume 153 of {\em
  North-Holland Mathematics Studies}.
\newblock North-Holland Publishing Co., Amsterdam, 1988.
\newblock Notas de Matem\'atica [Mathematical Notes], 122.

\bibitem[Can17a]{can17compact}
Pablo Candela.
\newblock Notes on compact nilspaces.
\newblock {\em Discrete Analysis}, 16, 2017.

\bibitem[Can17b]{can17algebraic}
Pablo Candela.
\newblock Notes on nilspaces - algebraic aspects.
\newblock {\em Discrete Analysis}, 15, 2017.

\bibitem[dV93]{dV93}
J.~de~Vries.
\newblock {\em Elements of topological dynamics}, volume 257 of {\em
  Mathematics and its Applications}.
\newblock Kluwer Academic Publishers Group, Dordrecht, 1993.

\bibitem[EE14]{ellis2014automorphisms}
David~B Ellis and Robert Ellis.
\newblock {\em Automorphisms and Equivalence Relations in Topological
  Dynamics}, volume 412.
\newblock Cambridge University Press, 2014.

\bibitem[EG60]{EG60}
Robert Ellis and W.~H. Gottschalk.
\newblock Homomorphisms of transformation groups.
\newblock {\em Trans. Amer. Math. Soc.}, 94:258--271, 1960.

\bibitem[EGS75]{ellis1975proximal}
Robert Ellis, Shmuel Glasner, and Leonard Shapiro.
\newblock Proximal-isometric ({P} {J}) flows.
\newblock {\em Advances in Mathematics}, 17(3):213--260, 1975.

\bibitem[EK71]{ellis1971characterization}
Robert Ellis and Harvey Keynes.
\newblock A characterization of the equicontinuous structure relation.
\newblock {\em Transactions of the American Mathematical Society},
  161:171--183, 1971.

\bibitem[Fur61]{furstenberg1961strict}
Hillel Furstenberg.
\newblock Strict ergodicity and transformation of the torus.
\newblock {\em American Journal of Mathematics}, 83(4):573--601, 1961.

\bibitem[Gla76]{glasner1976proximal}
Shmuel Glasner.
\newblock {\em Proximal flows}.
\newblock Springer, 1976.

\bibitem[Gla94]{G94}
Eli Glasner.
\newblock Topological ergodic decompositions and applications to products of
  powers of a minimal transformation.
\newblock {\em J. Anal. Math.}, 64:241--262, 1994.

\bibitem[Gla03]{G03}
Eli Glasner.
\newblock {\em Ergodic theory via joinings}, volume 101 of {\em Mathematical
  Surveys and Monographs}.
\newblock American Mathematical Society, Providence, RI, 2003.

\bibitem[GMU08]{glasner2008metrizable}
Eli Glasner, Michael Megrelishvili, and Vladimir~V Uspenskij.
\newblock On metrizable enveloping semigroups.
\newblock {\em Israel Journal of Mathematics}, 164(1):317--332, 2008.

\bibitem[GMV16a]{GMVI}
Yonatan Gutman, Freddie Manners, and P{\'e}ter~P. Varj{\'u}.
\newblock The structure theory of nilspaces {I}.
\newblock Preprint. arxiv.org/abs/1605.08945, 2016.

\bibitem[GMV16b]{GMVIII}
Yonatan Gutman, Freddie Manners, and P{\'e}ter~P. Varj{\'u}.
\newblock The structure theory of nilspaces {III}: Inverse limit
  representations and topological dynamics.
\newblock Preprint. arxiv.org/abs/1605.08950, 2016.

\bibitem[GMV18]{GMVII}
Yonatan Gutman, Freddie Manners, and P{\'e}ter~P. Varj{\'u}.
\newblock The structure theory of nilspaces {II}: Representation as
  nilmanifolds.
\newblock To appear in Transactions of the American Mathematical Society.
  arxiv.org/abs/1605.08948, 2018.

\bibitem[GT10]{GT10}
Ben Green and Terence Tao.
\newblock Linear equations in primes.
\newblock {\em Ann. of Math. (2)}, 171(3):1753--1850, 2010.

\bibitem[HK05]{HK05}
Bernard Host and Bryna Kra.
\newblock Nonconventional ergodic averages and nilmanifolds.
\newblock {\em Ann. of Math. (2)}, 161(1):397--488, 2005.

\bibitem[HK08]{HK08}
Bernard Host and Bryna Kra.
\newblock Parallelepipeds, nilpotent groups and {G}owers norms.
\newblock {\em Bull. Soc. Math. France}, 136(3):405--437, 2008.

\bibitem[HKM10]{HKM10}
Bernard Host, Bryna Kra, and Alejandro Maass.
\newblock Nilsequences and a structure theorem for topological dynamical
  systems.
\newblock {\em Adv. Math.}, 224(1):103--129, 2010.

\bibitem[McM76]{mcmahon1976}
Douglas McMahon.
\newblock Weak mixing and a note on a structure theorem for minimal
  transformation groups.
\newblock {\em Illinois J. Math.}, 20(2):186--197, 06 1976.

\bibitem[McM78]{mcmahon1978relativized}
Douglas~C McMahon.
\newblock Relativized weak disjointness and relatively invariant measures.
\newblock {\em Transactions of the American Mathematical Society},
  236:225--237, 1978.

\bibitem[MKS66]{MKS66}
Wilhelm Magnus, Abraham Karrass, and Donald Solitar.
\newblock {\em Combinatorial group theory: {P}resentations of groups in terms
  of generators and relations}.
\newblock Interscience Publishers [John Wiley \& Sons, Inc.], New
  York-London-Sydney, 1966.

\bibitem[Par70]{parry1970dynamical}
William Parry.
\newblock Dynamical systems on nilmanifolds.
\newblock {\em Bulletin of the London Mathematical Society}, 2(1):37--40, 1970.

\bibitem[Par81]{P81}
William Parry.
\newblock {\em Topics in ergodic theory}, volume~75 of {\em Cambridge Tracts in
  Mathematics}.
\newblock Cambridge University Press, Cambridge-New York, 1981.

\bibitem[Sep07]{Sep07}
Mark~R. Sepanski.
\newblock {\em Compact {L}ie groups}, volume 235 of {\em Graduate Texts in
  Mathematics}.
\newblock Springer, New York, 2007.

\bibitem[SY12]{SY12}
Song Shao and Xiangdong Ye.
\newblock Regionally proximal relation of order {$d$} is an equivalence one for
  minimal systems and a combinatorial consequence.
\newblock {\em Adv. Math.}, 231(3-4):1786--1817, 2012.

\bibitem[Sze12]{S12}
Balazs Szegedy.
\newblock On higher order fourier analysis.
\newblock Preprint. http://arxiv.org/abs/1203.2260, 2012.

\bibitem[Tao15]{TaoBlog15}
Terence Tao.
\newblock A nonstandard analysis proof of {S}zemeredi's theorem.
\newblock Blog post.
  https://terrytao.wordpress.com/2015/07/20/a-nonstandard-analysis-proof-of-szemeredis-theorem/,
  20 July, 2015.

\bibitem[TY13]{tu2013dynamical}
Siming Tu and Xiangdong Ye.
\newblock Dynamical parallelepipeds in minimal systems.
\newblock {\em Journal of Dynamics and Differential Equations}, 25(3):765--776,
  2013.

\bibitem[Vee68]{veech1968equicontinuous}
William~A Veech.
\newblock The equicontinuous structure relation for minimal abelian
  transformation groups.
\newblock {\em American Journal of Mathematics}, pages 723--732, 1968.

\bibitem[Vee77]{V77}
W.~Veech.
\newblock Topological dynamics.
\newblock {\em Bull. Amer. Math. Soc.}, 83(5):775--830, 1977.

\bibitem[Wei00]{weiss2000survey}
Benjamin Weiss.
\newblock A survey of generic dynamics.
\newblock {\em Descriptive set theory and dynamical systems (Marseille-Luminy,
  1996)}, pages 273--291, 2000.

\bibitem[Zie07]{Z07}
Tamar Ziegler.
\newblock Universal characteristic factors and {F}urstenberg averages.
\newblock {\em J. Amer. Math. Soc.}, 20(1):53--97 (electronic), 2007.

\end{thebibliography}

\end{document}